\documentclass[12pt]{article}

\usepackage{graphicx}
\usepackage{latexsym,amssymb}
\usepackage{amsthm}
\usepackage{indentfirst}
\usepackage{amsmath}
\usepackage{color}
\usepackage{xcolor}
\usepackage{fourier}
\usepackage[colorlinks=true,backref=page]{hyperref}
\usepackage[all]{xy}
\usepackage[overload]{empheq}

\usepackage[a4paper,text={160true mm,230true mm},centering,top=35true mm,head=5true mm,headsep=2.5true mm,foot=8.5true mm]{geometry}

\newtheorem{theoremalph}{Theorem}

\newtheorem*{Main Theorem}{Main Theorem}
\newtheorem{Coro}[theoremalph]{Corollary}
\newtheorem{Theorem}{Theorem}[section]
\newtheorem*{Theorem A}{Theorem A}
\newtheorem*{Theorem A'}{Theorem A'}
\newtheorem*{Theorem B'}{Theorem B'}

\newtheorem*{Conjecture}{Conjecture}
\newtheorem{Definition}[Theorem]{Definition}
\newtheorem{Proposition}[Theorem]{Proposition}
\newtheorem{Lemma}[Theorem]{Lemma}

\newtheorem{Remark}[Theorem]{Remark}

\newtheorem{Remark-numbered}[Theorem]{Remark}
\newtheorem{Corollary}[Theorem]{Corollary}

\newtheorem*{Claim}{Claim}
\newtheorem{Claim-numbered}[Theorem]{Claim}

 \def\NN{{\mathbb N}} 

 \def\RR{{\mathbb R}}

 \def\ZZ{{\mathbb Z}}

    \def\cU{{\cal U}}
   \def\cP{{\cal P}}

\begin{document}

\title{Effective SPR property for surface diffeomorphisms and three-dimensional vector fields}
	
\author{David Burguet, Chiyi Luo and Dawei Yang\footnote{
	D. Yang  was partially supported by National Key R\&D Program of China (2022YFA1005801), NSFC 12171348 \& NSFC 12325106, ZXL2024386 and Jiangsu Specially Appointed Professorship. C. Luo was partially supported by NSFC 12501244.}}
\date{}    
	
\maketitle
		
\begin{abstract}
In this paper, we prove that ergodic measures with large entropy give uniformly large measure to the set of points with simultaneously long unstable and long stable manifolds.  
As a consequence, for $C^{\infty}$ surface diffeomorphisms, we establish an effective version of the SPR property. 
For $C^{\infty}$ three-dimensional flows without singularities, we prove the finiteness of equilibrium measures for admissible potentials whose variation is strictly less than half of the topological entropy. 
\end{abstract}
	   
\tableofcontents

\clearpage
       
\section{Introduction}\label{SEC: Introduction} 
To study smooth ergodic theory of diffeomorphisms, the ``homoclinic relation'' has recently played an important role.
The homoclinic relation between hyperbolic periodic orbits were first defined by Newhouse \cite{New72}. 
This notion has been extended to general hyperbolic ergodic measures; see, for instance, Buzzi-Crovisier-Sarig \cite[Section 2.4]{BCS22}.
They proved that homoclinically related hyperbolic ergodic measures can be coded in one transitive countable Markov shift and established the local uniqueness of measures of maximal entropy.
By controling the number of homoclinic classes via a dynamical Sard lemma, they confirmed Newhouse's conjecture \cite{New91} on the finiteness of MMEs for $C^{\infty}$ surface diffeomorphisms.
 
The sizes of stable and unstable manifolds provide direct information for establishing homoclinic relations, without using the dynamical Sard lemma.
For a given hyperbolic ergodic measure, the existence of stable and unstable manifolds follows from Pesin’s theory \cite{BP07}.
However, the sizes of these manifolds depend on the measure, which does not suffice to establish the statistical properties for certain important hyperbolic ergodic measures.
Buzzi-Crovisier-Sarig \cite{BCS25} proved that for measures that are close (qualitatively) to measures of maximal entropy, in the setting of $C^\infty$ surface diffeomorphisms with positive topological entropy, the Pesin blocks can have uniformly large measure\footnote{This gives uniform lower bounds on the sizes of stable and unstable manifolds, independent of the measures.}, and then they proved various statistical properties of measures of maximal entropy. 
This property is called strongly positive recurrence (SPR for short).
 
In this paper, we establish an effective SPR property for surface diffeomorphisms, along with several extended theorems. 
Specifically, we obtain in a quantitative way uniform sizes of stable and unstable manifolds simultaneously for all ergodic measures with large enough metric entropy. 
Let us point out  that:
\begin{itemize}
\item We do not rely on the continuity of Lyapunov exponents.
            This allows us to prove the SPR property for $C^\infty$ surface diffeomorphisms without relying on the main results of the preparatory work \cite{BCS22I}.
            Moreover, we establish the continuity of Lyapunov exponents as a direct consequence of the SPR property.
            
\item Our arguments give a perturbative version, i.e., for diffeomorphisms that are close the original one.
 
\item It also applies to three-dimensional vector fields $X: M\rightarrow TM$ without singularities.  
            Our proof recovers Zang's result \cite{Zang25} for measures of maximal entropy, and established the finiteness of the number of homoclinic classes for ergodic measures whose metric entropy is uniformly larger than $ h_{\rm top}(X)/2$, not just those of measures of maximal entropy.
            Moreover, we prove that the number of general ergodic equilibrium states is finite for admissible potentials $\phi: M\rightarrow \RR$ satisfying either $\sup_{x\in M}|\phi(x)|< h_{\rm top}(X)/4$ or $\sup\{\phi(x)-\phi(y)\ :\ x,y\in M\}< h_{\rm top}(X)/2$.
\end{itemize}

\subsection{Long stable and unstable manifolds}

Pesin blocks are sets on which stable and unstable manifolds are uniformly long.
However, beyond Pesin theory, one may also obtain long stable and unstable manifolds at points which are typical for ergodic measures with large entropy. 
These observations enable us to establish homoclinic relations, for instance, Theorem \ref{Thm:F-H-F}.

Let $M$ be a compact, boundaryless $C^{\infty}$ Riemannian manifold of any dimension and let $f: M\rightarrow M$ be a $C^{r}, r>1$ diffeomorpism.
By Oseledec's theorem and Pesin's theory, there is an $f$-invariant subset $M'_{f}$ of full measure for all invariant measures,  such that for every $x\in M'_f$
$$T_xM=E^{s}(x,f) \oplus E^{c}(x,f) \oplus E^{u}(x,f),$$
where $E^{s/c/u}(x,f)$ corresponds to the subspace associated with negative/zero/positive Lyapunov exponents respectively.
When the diffeomorphism $f$ is fixed, we denote these spaces simply by $ E^{s/c/u}(x)$.
For each  $x\in M'_f$, the unstable manifold
\begin{equation}\label{eq: dofWu}
	W^u(x):=W^u(x,f)=\big\{y\in M: \limsup_{n\to\infty} \frac{1}{n} \log d( f^{-n}(y), f^{-n}(x) )<0 \big\}.
\end{equation}
and the stable manifold
\begin{equation}\label{eq: dofWs}
	W^s(x):=W^s(x,f)=\big\{y\in M: \limsup_{n\to\infty} \frac{1}{n} \log d( f^{n}(y), f^{n}(x) )<0 \big\}
\end{equation}
are $C^r$ immersed manifolds.   	
For $\beta>0$, let $E^{u/s}(x)(\beta):=\{v\in E^{u/s}(x): \| v \|\leq \beta \}$.
We consider the sets of points with unstable/stable manifolds of size $\beta$.  
\begin{Definition}\label{Def:long-unstable-stable}
	For $\beta\in (0,1)$, define the set of points with long unstable manifolds of size $\beta$
	\begin{align*}
		L^u_{\beta}(f)=&\{x\in M'_f:~\exists W_x\subset W^u(x),~\text{\rm s.t.}~\exp_x^{-1}W_x~\textrm{is a $C^r$ graph of a map~}\varphi:~E^{u}(x)\to (E^{u}(x))^\perp,\\
		{\rm Lip}&(\varphi)\le 1/3,~{\rm Domain}(\varphi)\supset E^{u}(x)(\beta),~\angle(T_y W_x,E^{u}(x))\le \frac{1}{\beta} \cdot d(x,y)^{\min\{1,r-1\}},~\forall y\in W_x\}.
	\end{align*}    
	and define the set of points with long stable manifolds of size $\beta$
	\begin{align*}
		L^s_{\beta}(f)=&\{x\in M'_f:~\exists W_x\subset W^s(x),~\text{\rm s.t.}~\exp_x^{-1}W_x~\textrm{is a $C^r$ graph of a map~}\varphi:~E^s(x)\to (E^s(x))^\perp,\\
		{\rm Lip}&(\varphi)\le 1/3,~{\rm Domain}(\varphi)\supset E^{s}(x)(\beta),~\angle(T_y W_x,E^s(x))\le \frac{1}{\beta}\cdot d(x,y)^{\min\{1,r-1\}},~\forall y\in W_x\}.
	\end{align*}  
\end{Definition}

Note that the whole manifold $M$ may have arbitrary dimension. 
When the unstable manifold is one-dimensional, we have the following result:
\begin{theoremalph}\label{Thm:unstable-measure-close-1}
	Let $f$ be a $C^\infty$ diffeomorphism with positive topological entropy. 
	Given $1>\alpha_1>\alpha_2>0$, there exists $\beta_0>0$ such that for any ergodic measure $\mu$ of $f$ with exactly one positive Lyapunov exponent, if $h_\mu(f)\geq \alpha_1 h_{\rm top}(f)$, then $\mu (L^u_{\beta_0}(f))>\alpha_2$.
\end{theoremalph}

Theorem~\ref{Thm:unstable-measure-close-1} provides more detailed informations than the main results in \cite{LuY24}. 
In \cite{LuY24}, even when the metric entropy of an ergodic measure is close to the topological entropy, one only knows that $\mu(L^u_{\beta_0}(f))$ is bounded away from zero; one does not know whether it can be close to $1$.
However, to establish the homoclinic relation, it is preferable that $\alpha_2$ is close to $1$; at the very least, $\alpha_2$ should be greater than $1/2$. 
This fact will be applied to prove the finiteness of the number of homoclinic  classes (the equivalence classes under homoclinic relation of all hyperbolic ergodic measures, see Section~\ref{SEC:Homoclasses}).

To get the uniform size of stable and unstable manifolds simultanously, by adding some additional information in Theorem~\ref{Thm:unstable-measure-close-1}, one has the following Theorem~\ref{Thm:both-stable-unstable}.
\begin{theoremalph}\label{Thm:both-stable-unstable}
	Assume that $f$ is a $C^\infty$ diffeomorphism with positive topological entropy. 
	Then, for every $1>\alpha_1>\alpha_2>\frac{1}{2}$, there exist $\beta_0>0$, $C>0$ such that for any ergodic measure $\mu$ of $f$ with exactly one positive Lyapunov exponent and exactly one negative Lyapunov exponent, if $h_\mu(f)\geq \alpha_1 h_{\rm top}(f)$, then 
	$$\mu \Big(L^u_{\beta_0}(f)\cap L^s_{\beta_0}(f)\cap \big\{x: \angle\big(E^u(x),E^s(x)\big)>C \big\}\Big)>2\alpha_2-1.$$
\end{theoremalph}   

We expect that a similar statement should also hold under the weaker assumption $\alpha_2>0$, rather than only in the case $\alpha_2>\tfrac{1}{2}$.
\begin{Conjecture}
	Assume that $f$ is a $C^\infty$ surface diffeomorphism with positive topological entropy. 
	Then, for every $1>\alpha_1>\alpha_2>0$, there exist $\beta_0>0$, $C>0$ such that for any ergodic measure $\mu$ of $f$, if $h_\mu(f)\geq \alpha_1 h_{\rm top}(f)$, then 
	$$\mu \Big(L^u_{\beta_0}(f)\cap L^s_{\beta_0}(f)\cap \big\{x: \angle\big(E^u(x),E^s(x)\big)>C \big\}\Big)>\alpha_2.$$
\end{Conjecture}

The information on unstable manifolds (nonlinear conditions) also yields information about their derivatives (linear conditions).
For $1 > \alpha_1 > \alpha_2 > 0$, define
\begin{equation}\label{eq:chi12}
	\chi(\alpha_1,\alpha_2):=h_{\rm top}(f) \cdot  \frac{\alpha_1-\alpha_2}{1-\alpha_2}.
\end{equation}
\begin{theoremalph}\label{Thm:unstable-derivetive}
	Let $f$ be a $C^\infty$ diffeomorphism with positive topological entropy. 
	Given $1>\alpha_1>\alpha_2>0$, for each $0<\chi_0<\chi(\alpha_1,\alpha_2)$ there exist $\beta_0>0$, $C>0$ and a $C^\infty$ neighborhood $\mathcal U$ of $f$, such that for any $g\in\mathcal U$, for any ergodic measure $\mu$ of $g$ with exactly one positive Lyapunov exponent, if $h_\mu(g)\geq \alpha_1 h_{\rm top}(f)$, then 
	$$\mu\left(L^u_{\beta_0}(g)\cap \big\{x:~\|Dg^{-n}|_{E^{u}(x,g)}\|\le Ce^{-\chi_0 n},~\forall n>0 \big\}\right)>\alpha_2.$$
\end{theoremalph}
It is clear that Theorem \ref{Thm:unstable-measure-close-1} follows from Theorem \ref{Thm:unstable-derivetive}.
When $\alpha_2$ is close to one and the stable manifold is also one-dimensional, a stronger result can be obtained; see  Theorem \ref{Thm:Quasi-Diffeo}.
These results can be applied to establish SPR properties for surface diffeomorphisms, or three-dimensional flows.
We first introduce the following sets, which we call quasi-hyperbolic sets.
\begin{Definition}\label{Def:Quasi-Hyperbolic}
	For $\chi>0$ and $\ell>0$, we denote by $\mathcal{H}^{\chi}_{\ell}(f)$ the set of points $x\in M$ satisfying the following conditions:
	\begin{itemize}
		\item for each $n\geq 0$, $\|D_xf^{-n}|_{E^{u}(x)}\|\leq \ell e^{-\chi n}$ and $\|D_xf^{n}|_{E^{s}(x)}\|\leq \ell e^{-\chi n}$;
		\item $\sin \angle \big(E^s(x), E^u(x)\big)\geq \ell^{-1}$.
	\end{itemize}
\end{Definition}

\begin{Remark}
   {\rm In Section~\ref{SEC:PESet}, we recall the definition of  a Pesin set.
		The measure of a Pesin set is closely related to the measure of the corresponding quasi-hyperbolic set for every hyperbolic ergodic measures $\mu$ (see \cite[Lemma 2.20, Proposition 2.21]{BCS25}),
		$$\mu({\rm PES}_{\ell}^{\chi,\varepsilon}(f))>1-\frac{C(f,\chi)}{\varepsilon} (1-\mu(\mathcal{H}^{\chi}_{\ell}(f))),$$
		where $C(\chi,f)=\max \Big\{10\big(\chi+\log \|Df\|_{\sup}+\log \|Df^{-1}\|_{\sup} \big),~1 \Big\}$. }
\end{Remark}

\begin{theoremalph}\label{Thm:Quasi-Diffeo}
	Assume that $f$ is a $C^\infty$ diffeomorphism with positive topological entropy. 
	Then, for every $1>\alpha_1>\alpha_2>\frac{1}{2}$ and every $0<\chi_0<\chi(\alpha_1,\alpha_2)$, there are $\beta_0>0$, $C>0$ and a $C^\infty$ neighborhood $\mathcal U$ of $f$ such that
	\begin{itemize}
		\item for any $g\in\mathcal U$, for any ergodic measure $\mu$ of $g$ with exactly one positive Lyapunov exponent and exactly one negative Lyapunov exponent, if $h_\mu(g)\geq \alpha_1 h_{\rm top}(f)$, then 
		$$\mu\big(L^u_{\beta_0}(g)\cap L^s_{\beta_0}(g)\cap \mathcal{H}^{\chi_0}_{C}(g)\big) >2\alpha_2-1.$$
	\end{itemize}	 
\end{theoremalph}   

Clearly, Theorem \ref{Thm:both-stable-unstable} is a direct consequence of Theorem \ref{Thm:Quasi-Diffeo}. 
Therefore, to complete the proofs of all the results in this subsection, it suffices to prove Theorems  \ref{Thm:unstable-derivetive} and  \ref{Thm:Quasi-Diffeo}.

\subsection{Surface diffeomorphisms} \label{SEC:SUR}
 
    The main theorems above can be applied to the surface case. 
    Theorem~\ref{Thm:SPR-Diff} gives a quantitative and perturbative version of strongly positively recurrent property for surface diffeomorphisms.
    The notation ${\rm PES}_{\ell}^{\chi,\varepsilon}$ will be introduced in Section \ref{SEC:PESet}.
  
    \begin{theoremalph}\label{Thm:SPR-Diff}
 	 Let $f: M\rightarrow M$ be a $C^{\infty}$ surface diffeomorpism with positive topological entropy. 
 	 Then, there exist constants $\chi>0, C_{\chi}>0$, such that for every $\varepsilon>0$ and every $0<\tau<\frac{1}{4}$, there exist a constant $\ell>0$ and a $C^{\infty}$ neighborhood $\cU$ of $f$ satisfying 
     $$\mu ({\rm PES}_{\ell}^{\chi,\varepsilon}(g))>1-\frac{C_{\chi}}{\varepsilon} \cdot \tau$$ 
 	 for every $g \in \cU$ and for every ergodic measure $\mu$ of $g$ with $h_{\mu}(g)> (1-\tau) h_{\rm top}(f)$.
     \end{theoremalph}
 
     We recall below the definition of the SPR property introduced by Buzzi-Crovisier-Sarig \cite{BCS25}.
     \begin{Definition}\label{Def:SPR}
 	  A diffeomorphism $f$ is said to be strongly positively recurrent (SPR), if there exists $\chi>0$, such that for each $\varepsilon>0$
 	   $$\lim_{h\rightarrow h_{\rm top}(f)} ~ \lim_{\ell \rightarrow \infty} ~\inf \big\{\mu({\rm PES}_{\ell}^{\chi,\varepsilon}(f)):~\mu~\text{is
 	   	ergodic and}~h_{\mu}(f)>h\big\} =1.$$
      \end{Definition}
      In fact, the SPR property defined in \cite[Definition 1.3]{BCS25} does not require the limit to approach $1$; it only requires the liminf to be positive. 
  
 \begin{Remark}[Direct consequence of Definition~\ref{Def:SPR}]
 	 Let $f$ be a $C^{1+}$ surface diffeomorphism that satisfies Definition~\ref{Def:SPR}. 
 	 Then, for every sequence of ergodic measures $\{\mu_n\}_{n>0}$ of $f$ with $\mu_n \rightarrow \mu$ and $h_{\mu_n}(f) \rightarrow h_{\rm top}(f)$, we have
 	$$\lambda^{+}(\mu_n) \rightarrow \lambda^{+}(\mu) \text{ (see also Remark \ref{Re:CofLE}) and } h_{\mu}(f)=h_{\rm top}(f) \text{ (by \cite[Corollary C]{LuY25}).}$$
 	 
 \end{Remark} 
  
As a direct corollary of  Theorem \ref{Thm:SPR-Diff}, we prove that for every $C^{\infty}$ surface diffeomorphism $f$ with $h_{\rm top}(f)>0$, there exist $\chi>0$ and $C_{\chi}>0$ such that for every $\varepsilon>0$ and every $0<\tau<\frac{1}{4}$
$$\lim_{\ell \rightarrow \infty} ~\inf \big\{\mu({\rm PES}_{\ell}^{\chi,\varepsilon}(f)):~\mu~\text{is ergodic and}~h_{\mu}(f)>(1-\tau)h_{\rm top}(f)\big\} \geq 1-\frac{C_\chi}{\varepsilon} \cdot \tau.$$
In particular, letting $\tau \rightarrow 0$, we get that every $C^{\infty}$ surface diffeomorphism with positive topological entropy is SPR.
This provides another proof of \cite[Theorem A]{BCS25}, as our proof does not rely on the continuity of Lyapunov exponents in the preparatory work \cite{BCS22I}.
The key ingredient in our proof is a new upper bound for the metric entropy, expressed as a linear combination involving the topological entropy and the Lyapunov exponents, where the coefficient of the topological entropy is the measure of a set with good hyperbolicity, see \eqref{eq:key}. To show this upper bound we use a variant of Yomdin's theory \cite{Yom87,Gro87} introduced in \cite{Bur12}.

\subsection{Three dimensional vector fields}
The suspension of a $d$-dimensional diffeomorphism yields a $(d+1)$-dimensional flow.
However, flows may not admit a global cross section, so understanding their dynamics requires additional methods beyond those for diffeomorphisms.

Ergodic theory of surface diffeomorphisms has been extensively studied by Sarig \cite{Sar13} and Buzzi-Crovisier-Sarig \cite{BCS22}, \cite{BCS25}.
There are some progresses for 3-D flows, Gan-Yang \cite{GaY18}, Lima-Sarig \cite{LS19}, Buzzi-Crovisier-Lima \cite{BCL23} and Zang \cite{Zang25}.
However, in contrast to surface diffeomorphisms, the corresponding theory for three-dimensional flows is not yet well established. 
The main theorems presented above provide new progress in this direction.
For a flow $(\varphi^t)_{t\in\mathbb R}$ generated by a vector field $X:M\rightarrow TM$ over a three-dimensional compact Riemannian manifold $M$ without boundary, one can also define the homoclinic equivalent class ${\rm HEC}(\nu)$ of an hyperbolic measure $\nu$ (see Section \ref{SEC:Homoclasses}) and Pesin uniform sets (see Section~\ref{SEC:PESet}).
We first prove the finiteness of homoclinic equivalent classes with large entropy in the following setting.

\begin{theoremalph}\label{Thm:F-H-F}
	Suppose that $M$ is a three-dimensional compact Remiannian manifold without boundary. 
	Let $X:M \to TM$ be a $C^\infty$ vector field without singularities and let $(\varphi^t)_{t\in\mathbb R}$ be the flow generated by $X$.
	Then, for any $\tau>0$ we have
	$$\# \Big\{{\rm HEC}(\nu): \nu~\text{is an ergodic measure with}~h_{\nu}(X)>\left(\frac{1}{2}+\tau \right) h_{\rm top}(X) \Big\}<\infty.$$
\end{theoremalph}
Recall that, for every ergodic hyperbolic measure $\mu$ of $(\varphi^t)_{t\in\mathbb R}$, there is at most one ergodic hyperbolic measure $\nu$ which is homoclinically related to $\mu$ that maximizes the entropy, or, more generally, maximizes the topological pressure $P_{X}(\cdot)$ for admissible potentials such as Hölder continuous functions or geometric potentials \cite[Corollary 1.2]{BCL23}.
Thus, we can directly obtain the following corollary, since the role of \eqref{eq:pre} is to guarantee that every equilibrium measure has
entropy uniformly larger than $h_{\rm top}(X)/2$.
\begin{Coro}\label{Cor:eq}
	Suppose that $M$ is a three-dimensional compact Remiannian manifold without boundary. 
	Let $X:M\to TM$ be a $C^\infty$ vector field without singularities, $(\varphi^t)_{t\in\mathbb R}$ be the flow generated by $X$, and let $\phi$ be an admissible potential. Assume also that
	\begin{equation}\label{eq:pre}
		\sup \big\{\phi(x)-\phi(y):x\in M,~y\in M \big\} < \frac{h_{\rm top}(X)}{2}.
	\end{equation}
	Then, $\phi$ has at most finitely many ergodic equilibrium measures.
\end{Coro}
Note that measures of maximal entropy are the equilibrium measures for the potential $\phi \equiv 0$.
Thus, Corollary~\ref{Cor:eq} implies the main result of Zang~\cite{Zang25}.
The final theorem provides an estimate for the measure of Pesin set for three-dimensional vector fields without singularities, thereby establishing an effective SPR property for such vector fields.

\begin{theoremalph}\label{Thm:SPR-Flow}
  Let $X:M\to TM$ be a $C^\infty$ three-dimensional  vector field without singularities and let $(\varphi^t)_{t\in\mathbb R}$ be the flow generated by $X$.
  Then, there exists a constant $\chi>0$ and a constant $C_{\chi}>0$, such that for every $\varepsilon>0$ and every $0<\tau<\frac{1}{4}$, there exists $\ell>0$ such that for every ergodic measure $\mu$ with $h_{\mu}(X)>(1-\tau)h_{\rm top}(X)$, one has
 $$\mu ({\rm PES}_{\ell}^{\chi,\varepsilon}(X))>1-\frac{C_{\chi}}{\varepsilon}\cdot \tau.$$
\end{theoremalph}
     
\section{Preliminaries}\label{Sec:B-Y}
    
    \subsection{Pesin sets} \label{SEC:PESet}
     We first recall the definition of Pesin sets\footnote{In general, the definition of Pesin sets is usually given in the non-uniformly hyperbolic setting, where $E^{c} = \{0\}$. However, to make it easier to study flows, we keep the central direction.  In subsequent applications, $E^c$ primarily corresponds to the flow direction.}.
     Let $M$ be a compact $C^{\infty}$ Riemannian manifold of any dimension, and let $f: M\rightarrow M$ be a $C^{1}$ diffeomorphism.    
     \begin{Definition}
    	   	For $\chi>0$, we denote by ${\rm NUP}_{\chi}(f)$ the set of points $x\in M$ with the following:
    	   	\begin{itemize}
    		   		\item for every $y\in \{f^k(x):k\in \mathbb Z\}$, 
    		   		$$T_y M=E^s(y)\oplus E^c(y)\oplus E^u(y),~~D_yfE^{\tau}(y)=E^{\tau}(f(y)),~\tau=s,c,u~;$$
    		   		\item for every $v\in E^s(x)\setminus\{0\}$, $\lim\limits_{n\rightarrow \pm \infty} \frac{1}{n} \log \|D_xf(v)\|<-\chi$;
    				\item for every $v\in E^c(x)\setminus\{0\}$, $\lim\limits_{n\rightarrow \pm \infty} \frac{1}{n} \log \|D_xf(v)\|=0$;
    		   		\item for every $v\in E^u(x)\setminus\{0\}$, $\lim\limits_{n\rightarrow \pm \infty} \frac{1}{n} \log \|D_xf(v)\|>\chi$;
    		   		\item for every $\tau_1,\tau_2 \in \{u,c,s\}$ with $ \tau_1 \neq \tau_2$, 
    		   		           $\lim\limits_{n\rightarrow \pm \infty} \frac{1}{n} \log \angle \big(E^{\tau_1}(f^n(x)), E^{\tau_2}(f^n(x))\big)=0$.
    		\end{itemize}
    \end{Definition}
    
     For every ergodic measure $\mu$ of $f$, if the non-zero Lyapunov exponents of $\mu$ lie outside the interval $[-\chi,\chi]$, then we have $\mu( {\rm NUP}_{\chi}(f) )=1$.
     We denote by ${\rm m}(A):=\min\{\|Av\|:\|v\|=1\}$ the co-norm of the linear operator $A$.
     \begin{Definition}
       	For $\chi>0$, $\ell\geq 1$ and $0<\varepsilon \ll \chi$, we denote by ${\rm PESC}_{\ell}^{\chi,\varepsilon}(f)$ the set of points $x\in {\rm NUP}_{\chi}$ such that for every $k\in \mathbb Z$
       	\begin{itemize}
       		\item $\|D_{f^k(x)}f^{-n}|_{E^u(f^k(x))}\| \leq \ell e^{|k|\varepsilon}e^{-(\chi-\varepsilon)n},~\|D_{f^k(x)} f^{n}|_{E^s(f^k(x))}\|\leq \ell e^{|k|\varepsilon}e^{-(\chi-\varepsilon)n},~\forall n>0$;
           \item $\ell^{-1} e^{-|k|\varepsilon}e^{-|n|\varepsilon}\le {\rm m}(D_{f^k(x)}f^n|_{E^c(f^k(x))})\le\|D_{f^k(x)}f^n|_{E^c(f^k(x))}\|\le \ell e^{|k|\varepsilon}e^{|n|\varepsilon},~\forall n\in \ZZ$;
       		\item $\sin \angle (E^{\tau_1} (f^k(x)), E^{\tau_2} (f^k(x)))\geq \ell^{-1}e^{-k\varepsilon}, ~\tau_1,\tau_2 \in \{u,c,s\},~\tau_1 \neq \tau_2$.
       	\end{itemize}
        And we denote ${\rm PES}_{\ell}^{\chi,\varepsilon}(f):=\big\{x\in {\rm PESC}_{\ell}^{\chi,\varepsilon}(f): E^c(x)=\{0\} \big\}$.
      \end{Definition}
      
      By classical Pesin's non-uniformly hyperbolic theory, for every ergodic hyperbolic measure $\mu$ of $f$ with $\mu({\rm NUP}_{\chi}(f))=1$ and every $0<\varepsilon\ll \chi$, one has $\mu\left(\bigcup_{\ell>0} {\rm PESC}_{\ell}^{\chi,\varepsilon}(f)\right)=1$.
      	      
      Recall Definition \ref{Def:Quasi-Hyperbolic} of the quasi-hyperbolic set.
      In the non-uniformly hyperbolic setting, the measure of a Pesin set is closely related to the measure of the corresponding quasi-hyperbolic set.
      Buzzi-Crovisier-Sarig showed the following result in \cite[Proposition 2.21]{BCS25}.
      \begin{Lemma}\label{Lem:Pes-QH}
      	Let $\mu$ be an ergodic hyperbolic measure of $f$.
      	Then, for each $\chi>0$, $\ell >1$ and for any $\varepsilon>0$, one has
      	$$\mu({\rm PES}_{\ell}^{\chi, \varepsilon}(f))>1- \frac{C(\chi,f)}{\varepsilon}\Big(1-\mu(\mathcal{H}^{\chi}_{\ell}(f))\Big),$$
      	where $C(\chi,f)=\max \Big\{10\big(\chi+\log \|Df\|_{\sup}+\log \|Df^{-1}\|_{\sup} \big),~1 \Big\}$.
      \end{Lemma}     
      \medskip      
      
     Pesin sets can also be defined in the setting of vector fields.
     Assume that  $X:M\rightarrow TM$ is a $C^1$ vector field without singularities and $(\varphi^t)_{t\in \RR}$ is the flow generated by $X$.
     Note that the Lyapunov exponent in the flow direction is zero.
     For $0<\varepsilon \ll  \chi $ and $\ell \geq 1$, we denote 
     $${\rm NUP}_{\chi}(X)={\rm NUP}_{\chi}(\varphi^1), ~~{\rm PES}_{\ell}^{\chi,\varepsilon}(X):={\rm PESC}_{\ell}^{\chi,\varepsilon}(\varphi^1).$$
     \begin{Definition}\label{Def:QuasiHyperflow}
     	We define the set $\mathcal{H}^{\chi,\varepsilon}_{\ell}(X)$ of points $x\in M$ satisfying the following conditions:
     	\begin{itemize}
     		\item for each $n\geq 0$, $\|D_x\varphi^{-n}|_{E^{u}(x)}\|\leq \ell e^{-\chi n}$ and $\|D_x\varphi^{n}|_{E^{s}(x)}\|\leq \ell e^{-\chi n}$;
     		\item for each $n\in \ZZ$, $\ell^{-1} e^{-|n|\varepsilon}\le {\rm m}(D_x\varphi^n|_{E^c(x)})\le\|D_x\varphi^n|_{E^c(x)}\|\le \ell e^{|n|\varepsilon}$;
     		\item $\| X(x) \| \geq \ell^{-1}$ and $\sin \angle \big(E^{\tau_1}(x), E^{\tau_2}(x)\big)\geq \ell^{-1}, ~\tau_1,\tau_2 \in \{u,c,s\},~\tau_1 \neq \tau_2$.
     	\end{itemize}
     \end{Definition}
     
     \begin{Lemma}\label{Lem:Pes-QH-Flows}
     	Let $\mu$ be an ergodic hyperbolic measure of $(\varphi^t)_{t\in \RR}$ which generated by a $C^1$ vector field $X$. 
     	Then, for each $\chi>0$, $\ell >1$ and for any $\varepsilon>0$, one has
     	$$\mu({\rm PES}_{\ell}^{\chi, \varepsilon}(X))>1- \frac{C(\chi,X)}{\varepsilon}\Big(1-\mu(\mathcal{H}^{\chi,\varepsilon}_{\ell}(X))\Big),$$
     	where $C(\chi,X)=\max \Big\{20 \big(\chi+\log \|D\varphi^1\|_{\sup}+\log \|D\varphi^{-1}\|_{\sup} \big),~1 \Big\}$.
     \end{Lemma}     
      
    The difference between Lemma \ref{Lem:Pes-QH-Flows} and \cite[Proposition 2.21]{BCS25} is the additional consideration of the center direction and the angle between subspaces.
    The estimate for the center direction can be obtained by using a similarly argument as in \cite[Proposition 2.21]{BCS25}.
    The estimate on the angle between subspaces follows from the fact that
      $$\lim_{n\rightarrow \pm \infty} \frac{1}{n} \log \angle \big(E^{\tau_1}(f^n(x)), E^{\tau_2}(f^n(x))\big)=0,~\tau_1,\tau_2 \in \{u,c,s\},~\tau_1 \neq \tau_2.$$
    We omit the proofs of of Lemma \ref{Lem:Pes-QH} and Lemma \ref{Lem:Pes-QH-Flows}, which can be found in \cite[Proposition 2.21]{BCS25} and \cite[Section 4.2]{Zang25}, respectively.
                  
    \subsection{Metric entropy and topological entropy}
    Let $f$ be a diffeomorphism on a compact Riemannian manifold $M$. Denote by $d$ the Riemannian distance on $M$. 
    For $x\in M$, $n\in \mathbb N$ and $\varepsilon>0$, define the $(n,\varepsilon,f)$-Bowen ball at $x$ by
    $$B_n(x,\varepsilon,f):=\{y\in M:~d(f^i(y),f^i(x))<\varepsilon,~\forall 0\leq i<n\}.$$
    We sometimes use the simplified notations $B_n(x,\varepsilon)=B_n(x,\varepsilon,f)$ when there is no confusion.   
    A subset $Y\subset M$ is said an $(n,\varepsilon)$-spanning set, if $M=\bigcup_{z\in Y}B_n(z,\varepsilon).$
    We denote by $r(f,n,\varepsilon)$ the minimal cardinality of all possible $(n,\varepsilon)$-spanning sets.
    The topological entropy is thus defined to be 
    $$h_{\rm top}(f):=\lim_{\varepsilon\to 0}\limsup_{n\to\infty}\frac{1}{n}\log r(f,n,\varepsilon).$$    
    For a probability measure $\mu$ (not necessarily invariant) and a finite partition $\mathcal P$, define the \emph{static entropy} of $\mu$:
    $$H_\mu(\mathcal P)=\sum_{P\in\mathcal P}-\mu(P)\log\mu(P)=\int -\log\mu(P(x)){\rm d}\mu(x).$$
    For the diffeomorphism $f$, one defines
    ${\mathcal P}^n:={\mathcal P}^{n,f}=\bigvee_{j=0}^{n-1}f^{-j}(\mathcal P)$.
    For an invariant measure $\mu$, the metric entropy of $\mu$ with respect to a partition $\mathcal P$ is
    $$h_\mu(f,\mathcal P)=\lim_{n\to\infty}\frac{1}{n}H_\mu({\mathcal P}^n);$$
    and the \emph{metric entropy} of $\mu$ is defined to be 
    $$h_\mu(f)=\sup \big\{h_\mu(f,\mathcal P): \cP~\text{is a finite partition} \big\}.$$   
    We introduce the following proposition, which follows from  \cite[Proposition 2.1, Proposition 2.2]{LuY24} or { Lemma 2 in \cite{Bur24P}}. 
    Let $\xi$ be a measurable partition subordinate to $W^u$ with respect to $f$ {(we refer to \cite{LeY85} for the definition  and the standard  properties of subordinate partitions)}.  
    \begin{Proposition}\label{Prop:Two Balls}
    	Let  $f: M\rightarrow M$ be a $C^{r},r>1$ surface diffeomorphism and $\mu$ be an ergodic measure.
    	Then,  there exists $K\subset M$ with $\mu(K)>0$, such that for every $x\in K$, every measurable set $\Sigma \subset W^{u}_{{\rm loc}}(x)$ with $\mu_{\xi(x)}(\Sigma\cap K)>0$, one has
    	\begin{equation}\label{eq:LimK}
    		h_{\mu}(f)\le\liminf_{{\rm Diam }(\cP)\to 0}\liminf_{n\to\infty}\frac{1}{n} \log \# \big\{P\in \cP^n: P\cap K \cap \Sigma \cap \xi(x)\neq \emptyset \big\}.
    	\end{equation}
    \end{Proposition}
    
    Assume that  $X:M\rightarrow TM$ is a $C^1$ vector field and $(\varphi^t)_{t\in \RR}$ is the flow generated by $X$.
    For an invariant measure $\mu$ of the flow $(\varphi^t)_{t\in \RR}$, we define the metric entropy of $\mu$ of $X$ by $h_{\mu}(X)=h_{\mu}(\varphi^1)$. 
    Let $\mu$ be an ergodic measure $\mu$ of the $(\varphi^t)_{t\in \RR}$. 
    Note that $\mu$ may not be an ergodic measure of $\varphi^1$.  
    But every ergodic component $\nu$ of $\mu$ with respect to $\varphi^1$ satisfies $h_{\nu}(\varphi^1)=h_{\mu}(X)$.
    
    \subsection{Homoclinic equivalent classes of ergodic hyperbolic measures}\label{SEC:Homoclasses}
    \subsubsection{Surface diffeomorphisms}
    We first recall the definition of the homoclinic relation for ergodic hyperbolic measures of surface diffeomorphisms (see \cite{BCS22}). 
    Let $f$ be a $C^r,r>1$  surface diffeomorphism. 
    For two ergodic hyperbolic measures $\mu_1$ and $\mu_2$, we say that $\mu_1$ is \textit{homoclinicaly related with} $\mu_2$, if there exist measurable sets $\Lambda_1,\Lambda_2$ with $\mu_1(\Lambda_1)>0$ and $\mu_2(\Lambda_2)>0$\footnote{By ergodicity we can require equivalently that $\mu_1(\Lambda_1)=1$ and $\mu_2(\Lambda_2)=1$.}
    such that for every $x\in \Lambda_1$ and every $y\in \Lambda_2$, there exist $n_1,n_2,k_1,k_2 \in \ZZ$ for which
    $$W^u(f^{n_1}(x)) \pitchfork W^s(f^{n_2}(y))\neq \emptyset,~W^s(f^{k_1}(x)) \pitchfork W^u(f^{k_2}(y))\neq \emptyset.$$
    Let $\mu$ be a hyperbolic ergodic measure, the \textit{homoclinic equivalent classes} of $\mu$ is defined by 
    $${\rm HEC}(\mu):= \big\{\nu:~\nu~\text{is an ergodic hyperbolic measure that is homoclinically related with}~\mu \big\}.$$
    Here the homoclinic equivalence classes form a subset of all ergodic hyperbolic measures.
    Recently, Buzzi-Crovisier-Sarig \cite{BCS22} proved for any $a>0$ the finiteness of the number of homoclinic classes for ergodic measures of  a $C^{\infty}$ surface diffeomorphism with entropy larger than $a$:
    
    \begin{Theorem}\label{Thm:finiteness}
    	Let $f$ be a  $C^r,r>1$ surface diffeomorphism, then for any $a>0$
    	$$\# \left\{{\rm HEC}(\mu):~\mu~\text{is a hyperbolic ergodic measure with}~h_{\mu}(f)>\frac{\lambda_{\min}(f)}{r}+a\right\}<\infty.$$
    	In particular, if $f$ is a $C^{\infty}$ surface diffeomorphism, then we have
    	$$\# \big\{{\rm HEC}(\mu):~\mu~\text{is a hyperbolic ergodic measure with}~h_{\mu}(f)>a\big\}<\infty.$$
    \end{Theorem}
    Recall the Definition \ref{Def:long-unstable-stable} of long unstable manifolds. 
    Long stable and unstable manifolds make it easier to build homoclinic relations.
    The following proposition follows directly from the definition of homoclinic classes and the compactness of the whole manifold.   
    Since its proof is analogous to that of Proposition~\ref{Prop:FHCFlow}, we omit the details.
    \begin{Proposition}\label{Prop:FHCDiff}
    	Let $f$ be a $C^2$ surface diffeomorphism.
    	For every $\beta>0$ and $\ell>0$, we have
    	$$\# \Big\{{\rm HEC}(\mu): \mu\big(L^s_{\beta}(f) \cap  L^u_{\beta}(f)\cap \{x:\sin \angle (E^s(x), E^u(x))\geq \ell^{-1}\}\big)>0 \Big\}<A_M \cdot \beta^{-2}\ell^{-6},$$
        for some constant $A_M$ depending only on $M$. 
      \end{Proposition}

    \subsubsection{Three-dimensional vector fields}
    Suppose that $M$ is a three-dimensional compact Remiannian manifold without boundary. 
    Let $X:M \rightarrow TM$ be a $C^{r},r>1$ vector field without singularities and $(\varphi^t)_{t\in \RR}$ be the flow generated by $X$. 
    One can define a notion of homoclinic relation between ergodic hyperbolic measures (see also \cite{BCL23}).
    For every regular point $x\in M$, we define the \textit{weak-stable} and \textit{weak-unstable} manifold at $x$ by 
    $$W^s({\rm Orb}(x))=\bigcup_{t\in \mathbb{R}} W^s(\varphi^t(x)),~W^u({\rm Orb}(x))=\bigcup_{t\in \mathbb{R}} W^u(\varphi^t(x))$$
    where $W^{u/s}(x):=W^{u/s}(x,\varphi^1)$ are as defined in \eqref{eq: dofWu} and \eqref{eq: dofWs}, and are one-dimensional immersed sub-manifolds, respectively.
    Then, for two ergodic hyperbolic measures $\mu_1$ and $\mu_2$ of the flow $(\varphi^t)_{t\in \RR}$, we say that $\mu_1$ is \textit{homoclinically related with} $\mu_2$, if there exist measurable sets $\Lambda_1, \Lambda_2$ with $\mu_1(\Lambda_1)>0$ and $\mu_2(\Lambda_2)>0$ such that
    $$W^u({\rm Orb}(x)) \pitchfork W^s({\rm Orb}(y)) \neq \emptyset,~~W^u({\rm Orb}(y)) \pitchfork W^s({\rm Orb}(x)) \neq \emptyset,~~\forall x\in \Lambda_1,~\forall y\in \Lambda_2.$$
    Let $\mu$ be an ergodic hyperbolic measure of the flow $(\varphi^t)_{t\in \RR}$. 
    By \cite[Proposition 10.1]{BCL23}, the homoclinic relation is an equivalence relation among ergodic hyperbolic measures of $(\varphi^t)_{t\in \RR}$.
    The homoclinic equivalent classes of $\mu$ is then defined as
    $${\rm HEC}(\mu):=\big\{\nu:~\nu~\text{is an ergodic hyperbolic measure that is homoclinically related with}~\mu \big\}.$$

    For $x,y\in M$ with $d(x,y)$ less than the radius of injectivity $r(M)$ of $M$ and for subspaces $E\subset T_xM$ and $F\subset T_yM$, we let 
    $\angle (E,F):=\min \{\angle (E, D_y(\exp_x^{-1})(F)),  \angle (F, D_x(\exp_y^{-1})(E))\}$. 
    We also define the Lipschitz constant of $X$ as the least positive real number $L>0$ satisfying 
    $\|d_y(\exp_x^{-1}) X(y)- X(x)\|, \|d_x(\exp_y^{-1})X(x)- X(y)\|\leq Ld(x,y)$ for any $x,y$ with $d(x,y)<r(M)$.
    \begin{Proposition}\label{Prop:FHCFlow}
    	  Let $X:M \rightarrow TM$ be a $C^{r},r> 1$ vector field without singularities and $(\varphi^t)_{t\in \RR}$ be the flow generated by $X$. 
    	  For every $\beta>0$ and $\ell>0$, we have
    	\begin{align*}
    		\# \Big\{{\rm HEC}(\mu): \mu\big(L^s_{\beta}(\varphi^1) \cap  L^u_{\beta}(\varphi^1)\cap \big\{x:\angle(E^{\tau_1}(x),E^{\tau_2}(x))>\ell^{-1},~\tau_1\neq \tau_2\in \{u,c,s\} \big\} \big)>0 \Big\} & \\
        <A_M \left(\max\left\{ \frac{{\rm Lip}(X)}{\min_z\|X(z)\|\wedge 1},\beta^{-1}\right\}\right)^3\ell^{-9}&.        	
        \end{align*} 
        for some $A_M$ depending only on $M$, where $a \wedge b:=\min \{a,b\}$.
    \end{Proposition}
    \begin{proof} 
    Fix $\beta>0$ and $\ell>0$.
    Let $E^c(x)=\mathrm{Span}(X(x))$ be the flow direction at $x$.
    For every $x,y\in M$ with $d(x,y)<\beta_0(\ell):=\min\left\{r(M),~0.001(\ell\mathrm{Lip}(X))^{-1}(\min_z\|X(z)\|\wedge 1)\right\}$, we have in a local chart \begin{equation}\label{ECosc}
    \|X(x)-X(y)\|\leq 0.001\ell^{-1}\|X(x)\|  
    \end{equation}
    Recall the Definition \ref{Def:long-unstable-stable} of the set $L^{u/s}_{\beta}(\varphi^1)$. 
    For each $x\in L^{u}_{\beta}(\varphi^1)\cap L^{s}_{\beta}(\varphi^1)$, there exist curves $\sigma_u:[-1,1]\rightarrow M$ and $\sigma_s:[-1,1]\rightarrow M$ such that for each $\kappa\in \{u,s\}$, we have  $\sigma_{\kappa}(0)=x$ and the minimum of the lengths ${\rm Length}(\sigma_{\kappa}[0,1] )$ and ${\rm Length}(\sigma_{\kappa}[-1,0])$ is larger than $0.001\beta \ell^{-1}$. Moreover 
    \begin{align}\label{lef}\forall  s\in [-1,1], \ \|\sigma'_\kappa(s)-\sigma_\kappa'(0)\|&\leq 0.001\ell^{-1}\|\sigma'_\kappa(0)\|.
        \end{align}
    We let 
    $$\beta(\ell):=\min \{\beta_0(\ell),~0,0005\beta \ell^{-1}\}$$ and 
    $$L(\beta,\ell):= L^s_{\beta}(\varphi^1) \cap  L^u_{\beta}(\varphi^1)\cap \big\{x:\angle(E^\tau_1(x),E^\tau_2(x))>\ell^{-1},~\tau_1\neq \tau_2\in \{u,c,s\} \big\}.$$
    \begin{Claim}For every $x\in L(\beta,\ell)$, the following properties hold:
    	\begin{enumerate}
    		\item[(1)] there exists a $C^1$ map $\phi_{cu}^x:\{v=v_u+v_c: v_{\tau} \in E^{\tau}(x),~ \|v_{\tau}\|\leq \beta(\ell),~\tau=u,c\} \rightarrow E^s(x)$ such that $\phi_{cu}^x(0)=0$, ${\rm Lip}(\phi_{cu}^x)\leq 0.01$  and $\exp_x \big({\rm graph}(\phi_{cu}^x)\big)\subset W^u({\rm Orb}(x))$;
    		\item[(2)] there exists a $C^1$ map $\phi_{cs}^x:\{v=v_s+v_c: v_{\tau} \in E^{\tau}(x),~ \|v_{\tau}\|\leq \beta(\ell),~\tau=s,c\}\rightarrow E^u(x)$ such that $\phi_{cs}^x(0)=0$, ${\rm Lip}(\phi_{cs}^x)\leq 0.01$  and $\exp_x \big({\rm graph}(\phi_{cs}^x)\big)\subset W^s({\rm Orb}(x))$.
    	\end{enumerate}
        \end{Claim}
        \begin{proof}[Proof of the claim]We consider item (1),  item (2) being similar.
        Let $$\Phi:[-2\beta(\ell)/\min_z\|X(z)\|,2\beta(\ell)/\min_z\|X(z)\|]\times[-1,1]\to M $$ be defined as 
        $$\forall t,s , ~ \Phi(t,s)=\varphi^t(\sigma_u(s)).$$
        By the first variation equation of the flow  we have 
        $$\frac{\partial D_x\varphi^t}{\partial t}(v)=(D_{\varphi^t(x)}X)( D_x\varphi^t(v)).$$    Gr\"onwall lemma leads to the following estimates in a local chart
        \begin{align*}
        \|\partial_2\Phi(t,s)-\sigma_u'(s)\|&=\|D_{\sigma_u(s)}\varphi^t(\sigma_u'(s))-\sigma_u'(s)\|,\\
      &\leq (e^{\mathrm{Lip}(X)t}-1)\|\sigma_u'(s)\|,\\
      \|\partial_1\Phi(t,s)-X(\sigma_u(s))\|&=\|X(\Phi(t,s))-X(\sigma_u(s))\|,\\
      &\leq (e^{\mathrm{Lip}(X)t}-1)\|X_{\sigma_u(s)}\|.
      \end{align*}
      We get by using  (\ref{lef}) and (\ref{ECosc}) that
      \begin{align*}
       \|\partial_2\Phi(t,s)-\sigma_u'(0)\|&\leq \
      2t\mathrm{Lip}(X)\|\sigma_u'(0)\|+ (1+2t\mathrm{Lip}(X))0.001\ell^{-1}\|\sigma_u'(0)\|,\\
      &\leq 0.01\ell^{-1}\|\sigma'_u(0)\|,\\
      \|\partial_1\Phi(t,s)-X(x)\|&\leq 2t\mathrm{Lip}(X)\|X(x)\|+(1+2t\mathrm{Lip}(X))0.001\ell^{-1}\|X(x)\|,\\
      &\leq 0.01\ell^{-1}\|X(x)\|.  
      \end{align*}
       In particular we have $\angle (\partial_2\Phi(t,s),\sigma_u'(0))\leq 0.01\ell^{-1} $ and $\angle(\partial_1\Phi(t,s),X(x))\leq  0.01\ell^{-1}$. As $x$ belongs to $L(\beta, \ell)$, we have 
       $\angle(\sigma_u'(0),X(x))=\angle(E^u(x),E^c(x))> \ell^{-1}$ so that $\angle (T_x\mathrm{Im}(\Phi), E)\leq 0.01$ with $E:=\mathrm{span}(X(x),\sigma'_u(0))$.

       Let $\pi:\mathbb R^d\to E$ be the orthogonal projection. Then $\pi\circ \Phi$ is a diffeomorphism onto its image $\pi(E)\supset \{v=v_u+v_c: v_{\tau} \in E^{\tau}(x),~ \|v_{\tau}\|\leq \beta(\ell),~\tau=u,c\}$ and the image  of $\Phi$ is a graph over this set of a $0.01$-Lipschitz  map.   
\end{proof}
        
    	\begin{Claim}
    	For every $x,y\in L(\beta,\ell)$ with $d(x,y)<0.1\ell^{-1}\beta(\ell)$ and $\angle(E^{\tau}(x),E^{\tau}(y))<0.1\ell^{-1},~\tau \in \{u,c,s\}$, 
    	we have $W^u({\rm Orb}(x)) \pitchfork W^s({\rm Orb}(y)) \neq \emptyset$, $W^u({\rm Orb}(y)) \pitchfork W^s({\rm Orb}(x)) \neq \emptyset$.
    	\end{Claim}    	
    	\begin{proof}[Proof of the claim]
    		For every $x,y\in L(\beta,\ell)$ as in the claim, there exist a map
    		$$ \phi_{cs}^{y\to x}: \{v=v_s+v_c: v_{\tau} \in E^{\tau}(x),~ \|v_{\tau}\|\leq 0.9\beta(\ell),~\tau=s,c\} \rightarrow E^u(x) $$  	
            such that ${\rm Lip}(\phi_{cs}^{y\to x})\leq 0.02$, $y=\exp_x(v_s^y+v_c^y+\phi_{cs}^{y\to x}(v_s^y+v_c^y))$ for some $(v_s^y,v_c^y)\in E_s(x)\times E_{c}(x)$ with $\|v_s^y\|, \|v_c^y\|,\|\phi_{cs}^{y\to x}(v_s^y+v_c^y))\|\leq 0.1\beta(\ell)$ and  $\exp_x \big({\rm graph}(\phi_{cs}^{y\to x})\big)\subset \exp_y \big({\rm graph}(\phi_{cs}^{y})\big)$.
        
    		We show that ${\rm graph}(\phi_{cs}^{y\to x})  \pitchfork {\rm graph}(\phi_{cu}^{x})\neq \emptyset$, which implies $W^s({\rm Orb}(y)) \pitchfork W^u({\rm Orb}(x))\neq \emptyset$.
    		The other case can be proved similarly.
    		We consider the map $\Psi:=\Psi_{v_c^y}: E^u(x)(0.9\beta(\ell))\rightarrow E^u(x)$, $\Psi(v_u):= \phi_{cs}^{y \to x}(v_c^y+\phi_{cu}^x(v_c^y+v_u))$. 
    	    One has $\|\Psi(\phi_{cs}^{y\to x}(v_s^y+v_c^y))-\phi_{cs}^{y\to x}(v_s^y+v_c^y)\|\leq 0.02\|\phi_{cu}^x(v_c^y+v_u^y)-v_s^y\|\leq 0.1\beta(\ell)$,
            therefore $\|\Psi(\phi_{cs}^{y\to x}(v_s^y+v_c^y))\|\leq 0.2\beta(\ell)$.
	        On the other hand, for each $v_u,e_u\in E^u(x)(0.9\beta(\ell) )$ we get 
            $\|\Psi(v_u)-\Psi(e_u)\|\leq {\rm Lip} (\phi_{cs}^{y \to x}) \cdot {\rm Lip} (\phi_{cu}^{x})\cdot  \|v_u-e_u\|\leq 0.01 \cdot  \|v_u-e_u\|\leq 0.1\beta(\ell)$. 
    		Consequently, the map $\Psi$ is a contraction on  $E^u(x)(0.9\beta(\ell))$. 
    		Then, there is a unique point $v_{c}^{u}\in E^u(x)(0.9\beta(\ell))$ such that $\Psi(v_{c}^{u})=v_{c}^{u}$.
            Then $\exp_x(v_c^y+\phi_{cu}^x(v_c^y+v_c^u)+\phi_{cs}^{y \to x}(v_c^y+\phi_{cu}^x(v_c^y+v_u)) )\in \exp_x \big({\rm graph}(\phi_{cs}^{y\to x})\big)\subset W^s(\rm{Orb}(y))$, but also $\exp_x(v_c^y+\phi_{cu}^x(v_c^y+v_c^u)+\phi_{cs}^{y \to x}(v_c^y+\phi_{cu}^x(v_c^y+v_c^u)) )= \exp_x(v_c^y+\phi_{cu}^x(v_c^y+v_c^u)+v_c^u )\in \exp_x({\rm graph}(\phi_{cu}^x))\subset W^u(\rm{Orb}(x))$. 
            The intersection of  ${\rm graph}(\phi_{cs}^{y\to x})$ and ${\rm graph}(\phi_{cu}^x)$ being clearly transverse, we conclude $W^s({\rm Orb}(y)) \pitchfork W^u({\rm Orb}(x))\neq \emptyset$. 
    	\end{proof}   	
    	
        We finish now the proof of the proposition. 
        Let $\mathcal TM$ be the fiber bundle over $M$ given by triples $(x,E,F,G)$ with $E,F,G$ being one-dimensional subspaces of $T_xM$, $x\in M$. 
        Let $\delta$ as in the above claim.  
        By compactness of $\mathcal TM$, there is a finite subfamily $\mathcal F$ of $\mathcal T M$, such that for any $(x,E,F,G)\in \mathcal TM$ there is $(y,H,I,J)\in \mathcal F$ satisfying $\angle(E,H),\ \angle(F,I), \ \angle(G,J)<0.05\ell^{-1}$, $d(x,y)<0.05\ell^{-1}\beta(\ell)$. 
        We can choose $\mathcal F$ with $\sharp \mathcal F\leq C_M(\ell^{-1}\beta(\ell))^3\ell^{-3}$ with some constant $C_M$ depending only on $M$. 
        For any ergodic hyperbolic measure $\mu$ with $\mu(L(\beta,\ell))>0$ there is $\Lambda_\mu\subset L(\beta,\ell)$ with $\mu(\Lambda_\mu)>0$ and $(x_\mu, E_\mu,F_\mu, G_\mu)\in \mathcal F$ such that for any $y\in \Lambda_\mu$ we have $\angle(E_\mu,E^s(y)),\ \angle(F_\mu,E^u(y)), \ \angle(G_\mu,E^c(y))<0.05\ell^{-1}, \ d(x_\mu,y)<0.05\ell^{-1}\beta(\ell)$.
        Let $\mu_1, \cdots, \mu_{\sharp\mathcal F+1}$ be ergodic hyperbolic measures with $\mu_i(L(\beta, \ell))>0$ for any $i=1,\cdots,\sharp\mathcal F+1 $. 
        There are $i\neq j$ with $(x_{\mu_i}, E_{\mu_i},F_{\mu_i},G_{\mu_i})=(x_{\mu_j}, E_{\mu_j},F_{\mu_j}, G_{\mu_j})$. 
        In particular, for any $y_i,y_j\in \Lambda_{\mu_i}\times \Lambda_{\mu_j}$, we have  $\angle(E^\tau(y_i),E^\tau(y_j))<0.1\ell^{-1},~\tau \in \{u,c,s\}$ and $d(y_i,y_j)<0.1\ell^{-1}\beta(\ell)$. 
        By the above claim we get $W^u({\rm Orb}(y_i)) \pitchfork W^s({\rm Orb}(y_j)) \neq \emptyset$, $W^u({\rm Orb}(y_i)) \pitchfork W^s({\rm Orb}(y_j)) \neq \emptyset$. Thus $\mu_i$ and $\mu_j$ are homoclinicaly related and $\sharp \{{\rm HEC}(\mu), \ \mu(L(\beta, \ell))>0\}\leq \sharp \mathcal F$.
    \end{proof}
    \subsection{Raparametrization Lemmas} \label{SEC:Rapar}
    Let $M$ be a compact Riemannian manifold without boundary of any dimension. 
    Recall that a $C^r$ curve $\sigma:[-1,1]\rightarrow M$ is called bounded if 
    $$\sup_{2\leq s\leq r} \|D^s \sigma\|_{\sup} \leq \frac{1}{6} \|D\sigma\|_{\sup}.$$
    If $\sigma$ is a bounded curve, then $\|D\sigma\|_{\sup}\leq 2\|D_t\sigma \|$  and $\sin \angle(D_{s}\sigma, D_{t}\sigma) <\frac{1}{2}$ for any $t,s\in [-1,1]$.  
    A bounded curve is said to be strongly $\varepsilon$-bounded, if $\|D\sigma\|_{\rm sup}\le\varepsilon$.
    A map $\theta:[-1,1]\rightarrow [-1,1]$ is called a affine reparametrization if $\theta'$ is constant and positive.    
    Denote $\sigma_*:=\sigma([-1,1])$.
    \begin{Lemma}\label{Lem:2-Res}
    	 For every $\varepsilon>0$, there exists $\beta_{\varepsilon}>0$ such that for every bounded but not strongly $\varepsilon$-bounded curve $\sigma:[-1,1]\rightarrow M$, there exists $-1<t_1<0<t_2<1$ for which
    	 \begin{itemize}
    	 	\item for every $t \in [t_1,t_2]$, there exists a $C^r$ map  $\phi_{t}:\{ v\in T_{\sigma(t)}\sigma_{\ast}:\|v\|\leq  \beta_{\varepsilon}\}\rightarrow (T_{\sigma(t)}\sigma_{\ast})^{\perp}$ with $\phi_{t}(0)=0$, ${\rm Lip}(\phi_{t})\leq \frac{1}{3}$, and the graph $W:= \exp_{\sigma(t)} {\rm Graph}(\phi_{t})$  satisfies $W \subset \sigma([-1,1])$ and $\angle \big( T_{y} W, T_{\sigma(t)}W \big)\leq d(\sigma(t), y)/\beta_{\varepsilon}, \forall y\in W$;
    	 	\item let $\theta_{\sigma}^{-1}:[-1,1]\rightarrow [-1,t_1]$ and  $\theta_{\sigma}^{1}:[-1,1]\rightarrow [t_2,1]$ be two affine raparametrizations, then $\sigma \circ \theta_{\sigma}^{-1}$ and $\sigma \circ \theta_{\sigma}^{1}$ are strongly $\varepsilon$-bounded.
    	 \end{itemize} 
    \end{Lemma}
    \begin{proof}
    	   Let $t_1=\frac{ \varepsilon }{\|D\sigma\|_{\sup}}-1$ and $t_2=1-\frac{ \varepsilon }{\|D\sigma\|_{\sup}}$.
    	   Since $\sigma$ is bounded but not $\varepsilon$-strongly bounded, we have $-1<t_1<0<t_2<1$.
    	   The second item follows immediately from the choice of $t_1$ and $t_2$.
    	   
    	   We now prove the first term.
    	   By the definition of the bounded curve (see also \cite[Section 4]{Bur24}), for each $t\in [t_1,t_2]$ we have 
    	   $$\angle(T_{\sigma(t)} \sigma_{\ast},T_{\sigma(s)}\sigma_{\ast})\leq \frac{1}{4} |t-s| \leq  \frac{1}{2\varepsilon} d(\sigma(t),\sigma(s)),~\forall s\in [-1,1].$$
    	   Hence, one can choose $\beta \approx \varepsilon/2$, and the curve $\sigma_{*}$ contains the graph of $\phi_t$, which satisfies the conclusion of the first item.
    \end{proof}
     
    We have the following local reparametrization from \cite{Bur24P}.    
    \begin{Lemma}[\cite{Bur24P}, Lemma 12]\label{Lemma:local-reparametrization}
    	For $r\ge 2$, there is a constant $C_r>0$ with the following property.
    	Given $\Upsilon>0$, there is $\varepsilon:=\varepsilon(\Upsilon)>0$ such that if $g$ is a $C^r$ diffeomorphism with $\|Dg\|_{C^{r-1}}<\Upsilon$, then for any strongly $\varepsilon$-bounded $C^r$ curve $\sigma:[-1,1]\to M$ and any $\chi^+,\chi \in\mathbb Z$, there is a family of affine reparametrizations $\Theta$ such that
    	\begin{enumerate}
    		\item[(1)] $\{t\in[-1,1]:~x=\sigma(t),\lceil\log\|Dg(x)\|\rceil=\chi^+,~\lceil\log\|Dg|_{T_x{\sigma_*}}\|\rceil=\chi\}\subset \bigcup_{\theta\in\Theta}\theta([-1,1])$;
    		\item[(2)] $g\circ \sigma\circ\theta$ is bounded for any $\theta\in\Theta$;
    		\item[(3)] $|\theta'|\leq \frac{1}{2}  \exp(\frac{\chi-\chi^{+}-1}{r-1})$ for any $\theta\in\Theta$ and $\#\Theta\le C_r \exp(\frac{\chi^+-\chi}{r-1})$.
    	\end{enumerate}
    \end{Lemma}
    Let $\sigma:[-1,1]\rightarrow M$ be a bounded curve. 
    Let $g:M\rightarrow M$ be a $C^r$ diffeomorphism, we define the set
    $$\Sigma(\{\chi^{+}_k, \chi_k\}_{k=1}^{n})\subset \sigma_*$$
    as follows: $z \in \Sigma(\{\chi^{+}_k, \chi_k\}_{k=1}^{n})$ if and only if for every $1\le k \le n$
    $$\lceil \log\|D_{g^{k-1}(z)}g\| \rceil=\chi^+_k,~~\lceil \log\|D_{g^{k-1}(z)}g|_{T_{g^{k-1}(z)}{(g^{k-1}\circ \sigma)_*}}\| \rceil=\chi_k.$$
    Meanwhile, for any $z\in \sigma_*$, one defines
    $$\chi^+_{k}(z):=\lceil \log\|D_{g^{k-1}(z)}g\| \rceil,~~\chi_{k}(z):=\lceil \log\|D_{g^{k-1}(z)}g|_{T_{g^{k-1}(z)}{(g^{k-1}\circ \sigma_*)}}\| \rceil.$$

  \begin{Lemma}\label{Lem:Blue}
  	For every $\varepsilon>0$ and $\delta>0$,  if $\sigma$ is a $C^r$ strongly $\varepsilon$-bounded curve and $\{\theta_1,\cdots,\theta_{n}\}$ be a family of affine reparametrizations such that 
   \begin{enumerate}
   	\item[(1)] $g^{i}\circ \sigma \circ \theta_{1}\circ \cdots \circ \theta_{i}$ is strongly $\varepsilon$-bounded for every $i=1,\cdots,n-1$;
   	\item[(2)] $g^{n}\circ \sigma \circ \theta_{1}\circ \cdots \circ \theta_{n}$ is bounded, but not strongly $\varepsilon$-bounded;
   	\item[(3)] $(\sigma \circ \theta_{1} \circ \cdots \circ \theta_{n})_{\ast} \cap \Sigma(\{\chi^{+}_k, \chi_k\}_{k=1}^{n})\neq \emptyset$;
   	\item[(4)] $|\theta'_j|\leq 2^{-1}e^{-\delta}\exp(\frac{\chi_j-\chi^{+}_j-1}{r-1})$ for all $1\leq j \leq n$.
   \end{enumerate}
   Then, for every $x\in (\sigma \circ \theta_{1} \circ \cdots \circ \theta_{n})_{\ast} \cap \Sigma(\{\chi^{+}_k, \chi_k\}_{k=1}^{n})$, we have 
    $$ \log\|D_{g^{n-k}(x)}g^{k}|_{T_{g^{n-k}(x)}(g^{n-k}\circ \sigma)_{\ast}} \| >k\delta,~~\forall 1\leq k\leq n.$$    
   \end{Lemma}
   \begin{proof}
   Since  $g^{n}\circ \sigma \circ \theta_{1}\circ \cdots \circ \theta_{n}$ is not strongly $\varepsilon$-bounded, we must have
   $$\|D(g^{n}\circ \sigma \circ \theta_{1}\circ \cdots \circ \theta_{n}) \|_{\sup}>\varepsilon.$$
   For every $x\in (\sigma \circ  \cdots \circ \theta_n)_{\ast}\cap \Sigma(\{\chi^{+}_k, \chi_k\}_{k=1}^{n})$ and every $1\leq k \leq n$, we have 
   \begin{align*}
   	\varepsilon
    &<\|D(g^k\circ g^{n-k}\circ \theta_1\circ \cdots \circ \theta_{n-k} \circ \cdots \circ \theta_{n})\|_{\sup}\\
   	&\leq  2\|D_{g^{n-k}(x)}g^{k}|_{T_{g^{n-k}(x)}(g^{n-k}\circ \sigma)_{\ast}}\|\cdot \| D  g^{n-k}\circ \theta_1\circ \cdots \circ \theta_{n-k}\|_{\sup}\cdot \prod_{j=n-k+1}^{n}|\theta'_{j}| \\
   	&\leq2\varepsilon \|D_{g^{n-k}(x)}g^{k}|_{T_{g^{n-k}(x)}(g^{n-k}\circ \sigma)_{\ast}} \| \cdot 2^{-k} \cdot e^{-k\delta}\exp \big(\frac{\sum_{j=n-k+1}^{n} (\chi_j-\chi_j^{+}-1) }{r-1} \big).
   \end{align*}
   Note that 
   $$ \sum_{j=n-k+1}^{n} (\chi_j-\chi_j^{+}-1)\leq \log\|D_{g^{n-k}(x)}g^{k}|_{T_{g^{n-k}(x)}(g^{n-k}\circ \sigma)_{\ast}} \| -\sum_{j=n-k}^{n-1} \log \|D_{g^j(x)}g\|.$$
   Hence, we have
   $$2^{k-1}\cdot e^{k\delta}< \|D_{g^{n-k}(x)}g^{k}|_{T_{g^{n-k}(x)}(g^{n-k}\circ \sigma)_{\ast}} \|^{\frac{r}{r-1}} \cdot \prod_{j=n-k}^{n-1} \|D_{g^j(x)}g\| ^{\frac{-1}{r-1}}.$$
   Therefore, for every $ n\geq k\geq 1$ we have
   \begin{align*}
       &~\log\|D_{g^{n-k}(x)}g^{k}|_{T_{g^{n-k}(x)}(g^{n-k}\circ \sigma)_{\ast}} \| \\
      =&~\frac{r}{r-1} \left( \log\|D_{g^{n-k}(x)}g^{k}|_{T_{g^{n-k}(x)}(g^{n-k}\circ \sigma)_{\ast}} \| -\frac{1}{r} \sum_{j=n-k}^{n-1} \log\|D_{g^{j}(x)}g|_{T_{g^{j}(x)}(g^{j}\circ \sigma)_{\ast}} \| \right) \\
     \geq&~ \frac{r}{r-1} \left( \log\|D_{g^{n-k}(x)}g^{k}|_{T_{g^{n-k}(x)}(g^{n-k}\circ \sigma)_{\ast}} \| -\frac{1}{r}\sum_{j=n-k}^{n-1} \log \|D_{g^j(x)}g\| \right)\\
      > &~k\delta+(k-1)\log 2>k\delta.
   \end{align*}
   This completes the proof of this lemma.
   \end{proof}

   \section{Entropy and the measure of Pesin sets} \label{SEC:3}
   	 In this section, we provide the proofs of Theorems \ref{Thm:unstable-derivetive}, \ref{Thm:Quasi-Diffeo} and \ref{Thm:SPR-Diff}.
   	 We begin with a key proposition that serves as the main ingredient in the proofs of these theorems.
   	 We first introduce some notations. 
   	 
   	 Assume that $f:M\rightarrow M$ is a $C^r,~r>1$ diffeomorphism. 
   	 Given $1>\alpha_1>\alpha_2>0$ and  $0<\chi_0<\chi<\chi(\alpha_1,\alpha_2)$.
   	 Recall the choice of $\chi(\alpha_1,\alpha_2)$ in \eqref{eq:chi12}, we have
   	 $$ \frac{\alpha_1 h_{\rm top}(f)-\chi}{h_{\rm top}(f)-\chi}>\frac{\alpha_1 h_{\rm top}(f)-\chi(\alpha_1,\alpha_2)}{h_{\rm top}(f)-\chi(\alpha_1,\alpha_2)}=\alpha_2.$$
   	 Therefore, we can choose a small number $\alpha>0$ such that
   	 \begin{equation}\label{eq:choiceofalpha}
   	 	 \frac{\alpha_1 h_{\rm top}(f)-\chi-2\alpha}{h_{\rm top}(f)-\chi+\alpha}>\alpha_2.
   	 \end{equation} 
   	    	 
     Choose a large integer $q\in \NN$ such that (recall the number $C_r>1$ in  Lemma \ref{Lemma:local-reparametrization})
     \begin{equation} \label{eq: Q}
     \frac{2}{q}\log (5qC_r A_f)+\frac{1}{q(r-1)}<\alpha,~~\text{where}~A_f:=2\log(1+\|Df\|_{\sup}+\|Df^{-1}\|_{\sup}).
     \end{equation}
	 Choose a constant $\Upsilon:=\Upsilon_{q}>0$ and a $C^r$ neighborhood $\mathcal{U}_{1}$ of $f$ such that for any $g\in \mathcal{U}_1$, one has 
	 \begin{equation}\label{eq:ChooseUp}
	 	\max_{1\leq j \leq r} \{\|D^j g^{k}\|_{\sup},~\|D^j g^{-k}\|_{\sup} \}<\Upsilon,~\forall 1\leq k \leq q;~~4A_g \leq 5A_f,~~  \|Dg^{\pm}\|_{\sup}\leq 2 \|Df^{\pm}\|_{\sup}.
	 \end{equation} 	 	      
	 Choose $\varepsilon(\Upsilon)>0$ as in  Lemma \ref{Lemma:local-reparametrization}. 
	 Take a positive number $\varepsilon<\varepsilon(\Upsilon)/10$ and an integer $N \in \NN$, such that 
     \begin{equation} \label{eq:U}
     10\cdot \varepsilon \cdot \|Df\|_{\sup}^q<r(M) \text{ and}~~\forall n\geq N, \ \log r\big(f^q,n,\frac{\varepsilon}{4} \big) < q\cdot n \left(h_{\rm top}(f)+\frac{\alpha}{2}\right),
     \end{equation}
     where $r(f^q,n,\frac{\varepsilon}{4})$ denotes the minimal number of $(n,\frac{\varepsilon}{4},f^q)$-Bowen balls which can cover $M$, and $r(M)$ is a positive constant such that $\exp_x$ maps $T_xM(2r(M))$ diffeomorphically onto a neighborhood of $B(x,r(M))$ for every $x\in M$.
     There is a $C^r$ neighborhood $\mathcal{U}\subset \mathcal{U}_{1}$ of $f$ such that for every $g\in \mathcal{U}$    
     \begin{equation} \label{eq:U1}
     10\cdot \varepsilon \cdot \| Dg  \|_{\mathrm{sup}}^q<r(M) \  \text{and} \
       \forall n=N, \cdots, 2N, \ \log r\big(g^q,n,\frac{\varepsilon}{4}\big) <\log r\big(f^q,n,\frac{\varepsilon}{4}\big)+nq \cdot \frac{\alpha}{2}.
     \end{equation}
     Then, for any $n\geq N$ we get with $R=n-N(\lceil n/N\rceil -1)\in [N,2N]$
     \begin{equation}\label{eq:N}
     \begin{aligned}
     	\log r\big(g^q,n,\frac{\varepsilon}{2}\big)&\leq (\lceil n/N \rceil -1)\log r\big(g^q,N,\frac{\varepsilon}{4}\big)+\log r\big(g^q,R,\frac{\varepsilon}{4}\big) \\
     	&\leq q\cdot n (h_{\rm top}(f)+\alpha).
     \end{aligned} 
     \end{equation}

     Recall that $\chi<\chi(\alpha_1,\alpha_2)<\alpha_1 h_{\rm top}(f)$. 
     For $g\in \mathcal{U}$, we consider the sets
     \begin{align}
     	&H_{\chi}^u(g^{q}):= \{x : \| D_{x}g^{-qn}|_{E^{u}(x)} \|\leq e^{-qn\chi},~\forall n\geq 0 \}; \label{eq:lambda} \\
     	&H_{\chi}^s(g^{q}):= \{x : \| D_{x}g^{qn}|_{E^{s}(x)} \|\leq e^{-qn\chi},~\forall n\geq 0 \}.  \label{eq:lambda-s}
     \end{align}
     For each ergodic measure $\nu$ of $g^q$ with $h_{\nu}(g^q)\geq q\alpha_1 h_{\rm top}(f)>q\chi$, one has (see  \cite[Lemma 3]{Bur24P})
     $$\nu(H_{\chi}^u(g^{q}))>0,~~\nu(H_{\chi}^{s}(g^{q}))>0.$$
     Choose $\beta:=\beta_{\varepsilon}$ satisfies the conclusion of Lemma \ref{Lem:2-Res}.     
     Recall the Definition \ref{Def:long-unstable-stable} of $L^u_{\beta}(g)$ and $L^s_{\beta}(g)$.
     We denote
     \begin{align} 
     &\Lambda^{u,\chi}_{N,\beta}(g^q)=\{x: \exists \, 0\leq i\leq N~\text{s.t.}~g^{qi}(x)\in L^u_{\beta}(g)\cap H_{\chi}^{u}(g^q)\};  
     \label{eq:Lambda}  \\
     &\Lambda^{s,\chi}_{N,\beta}(g^q)=\{x: \exists \, 0\leq i\leq N~\text{s.t.}~g^{-qi}(x)\in L^s_{\beta}(g)\cap H_{\chi}^{s}(g^q)\}.  
     \label{eq:Lambda-s}
     \end{align}     
     We are going to prove that (which will be deduced from \eqref{eq:key})
     $$\mu(\Lambda^{u,\chi}_{N,\beta}(g^q))\gtrapprox \frac{h_{\mu}(g)-\chi}{h_{\rm top}(f) -\chi}.$$
     The main proposition of this section is stated as follows.
     
     \begin{Proposition} \label{Prop:KeyProp} 
     	Under the above notations, for every $g\in \mathcal{U}$ and every ergodic measure $\mu$ of $g$ with $h_{\mu}(g)\geq \alpha_1 h_{\rm top}(f)$, if $\mu$ has exactly one positive Lyapunov exponent, then we have
     	\begin{align*}
     \mu(\Lambda^{u,\chi}_{N,\beta}(g^q)) 
    \geq  		\frac{h_{\mu}(g)-\chi-\frac{1}{r-1}\big(\frac{1}{q}\int \log \|D_x g^q \|~{\rm d}\mu(x)-\lambda^{+}(\mu,g)\big)-\alpha}{h_{\rm top}(f)-\chi+\alpha};
     	\end{align*}
     	if $\mu$ has exactly one negative Lyapunov exponent, then we have
     	\begin{align*}
     		\mu(\Lambda^{s,\chi}_{N,\beta}(g^q)) 
     		\geq  		\frac{h_{\mu}(g)-\chi-\frac{1}{r-1}\big(\frac{1}{q}\int \log \|D_x g^{-q} \|~{\rm d}\mu(x)+\lambda^{-}(\mu,g)\big)-\alpha}{h_{\rm top}(f)-\chi+\alpha}.
     	\end{align*}
     \end{Proposition}
     
 \begin{Remark}
{\rm   
 	Since  $\chi_0<\chi$, one can choose $C:=C(N,\chi,\chi_0,q,\mathcal{U})$ such that  for every $g\in \mathcal{U}$
 	\begin{align} 
 		&H_{\chi_0,C}^u(g):=\{x: \|D_x g^{-n}|_{E^{u}(x)}\|\leq C e^{-n \chi_0},~\forall n\geq 0\} \supset \{x: \exists 0\leq i\leq N~\text{s.t.}~g^{qi}(x)\in H_{\chi}^u(g^q)\}; \label{eq:Keyset-u} \\
 		&H_{\chi_0,C}^s(g):=\{x: \|D_x g^{n}|_{E^{s}(x)}\|\leq C e^{-n \chi_0},~\forall n\geq 0\} \supset \{x: \exists 0\leq i\leq N~\text{s.t.}~g^{-qi}(x) \in H_{\chi}^s(g^q)\}.  \label{eq:Keyset-s}
 	\end{align} 
 	For each $x$ that lies at the center of a bounded curve of length large than $2\beta$, after at most $qN$ forward or backward iterations, it will lie at the center of a bounded curve whose length is approximately large than 
    $\frac{2\beta}{\|Dg^{\pm}\|_{\sup}^{2qN}}$.
 	Hence, by Lemma \ref{Lem:2-Res} there is $\beta_0:=\beta_0(\beta,N,q,\mathcal{U})>0$ such that for every $g\in \mathcal{U}$, one has
 	\begin{align} 
 	&	L_{\beta_0}^{u}(g) \supset \{x: \exists 0\leq i\leq N~\text{s.t.}~g^{qi}(x)\in L^u_{\beta}(g) \}; \label{eq:beta0} \\
 	&   L_{\beta_0}^{s}(g) \supset \{x: \exists 0\leq i\leq N~\text{s.t.}~g^{-qi}(x)\in L^s_{\beta}(g) \}. \label{eq:beta0-s}
 	\end{align}
 	Therefore, we have $\Lambda^{u,\chi}_{N,\beta}(g^q)\subset L_{\beta_0}^{u}(g) \cap H_{\chi_0,C}^u(g)$ and $\Lambda^{s,\chi}_{N,\beta}(g^q)\subset L_{\beta_0}^{s}(g) \cap H_{\chi_0,C}^s(g)$.           }
  \end{Remark}
  We first provide the proof of Theorem \ref{Thm:unstable-derivetive}, Theorem \ref{Thm:Quasi-Diffeo} and Theorem \ref{Thm:SPR-Diff} by assuming that Proposition \ref{Prop:KeyProp} holds, and then give the proof of Proposition \ref{Prop:KeyProp}.

\subsection{Proof of Theorem \ref{Thm:unstable-derivetive} and Theorem \ref{Thm:Quasi-Diffeo}}
The proof of Theorem~\ref{Thm:unstable-derivetive} essentially follows from Proposition \ref{Prop:KeyProp} and the discussion at the beginning of this Section. 
Below, we provide the complete proof.
\begin{proof}[Proof of Theorem~\ref{Thm:unstable-derivetive}]
	Assume that $f$ is a $C^\infty$ diffeomorphism with positive topological entropy. 
	Fix constants $1>\alpha_1>\alpha_2>0$ and $0<\chi_0<\chi<\chi(\alpha_1,\alpha_2)$.
	Let $\alpha > 0$ be the constant chosen in \eqref{eq:choiceofalpha}, we take $r>0$ sufficiently large so that
	\begin{equation}\label{eq:r}
		\frac{\log \big (2 \cdot \max \{ \|Df\|_{\sup}, \| Df^{-1}\|_{\sup} \}\big)}{r-1} \leq \alpha.
	\end{equation}
	Then, we choose a large number $q\in \NN$ satisfying \eqref{eq: Q}.
	Recall the definition of $\Lambda^{u,\chi}_{N,\beta}(g^q)$ in \eqref{eq:Lambda}. Let  $\mathcal{U}$ be  the $C^r$ neighborhood of $f$ as in \eqref{eq:U1}. We also choose  $\beta_0>0, C>0$ satisfying  the conditions in \eqref{eq:beta0} and  \eqref{eq:Keyset-u} respectively.
	Note that for every $g\in \mathcal{U}$, by \eqref{eq:r} one has
	$$\frac{1}{r-1}\left(\frac{1}{q}\int \log \|D_x g^q \|~{\rm d}\mu(x)-\lambda^{+}(\mu,g)\right)\leq \frac{ \log (2\|Df\|_{\sup})}{r-1}\leq \alpha.$$
	By Proposition \ref{Prop:KeyProp}, for any ergodic measure $\mu$ of $g$ with exactly one positive Lyapunov exponent and $h_\mu(g)\geq \alpha_1 h_{\rm top}(f)$, we have
	\begin{align*}
		\mu(H_{\chi_0,C}^u(g) \cap L_{\beta_0}^{u}(g)) =&~ \mu(L^u_{\beta_0}(g)\cap \{x:~\|Dg^{-n}|_{E^{u}(x)}\|\le Ce^{-\chi_0 n},~\forall n>0\}) \\
		\geq &~\mu( \Lambda^{u,\chi}_{N,\beta}(g^q))\\
		\geq  & ~\frac{h_{\mu}(g)-\chi-\frac{1}{r-1}\big(\frac{1}{q}\int \log \|D g^q \|~{\rm d}\mu-\lambda^{+}(\mu,g)\big)-\alpha}{h_{\rm top}(f)-\chi+\alpha} \\
		\geq &~\frac{h_{\mu}(g)-\chi-2\alpha}{h_{\rm top}(f)-\chi+\alpha} \\
		> & ~ \alpha_2  ~~~~~~~~~~~~\text{by \eqref{eq:choiceofalpha}} .
	\end{align*}
	This completes the proof of Theorem~\ref{Thm:unstable-derivetive}.
\end{proof}

Next, we prove Theorem~\ref{Thm:Quasi-Diffeo}.   
The main issue here is the angle between subspaces. 
We first prove the following lemma.
\begin{Lemma} \label{Lem:angles}
	Let $f$ be a $C^1$ diffeomorphism. 
	Assume that  $\lambda_1>\lambda_2$, $k>0$ and two $Df$-invariant measurable sub-bundles $E$ and $F$,  if $x$ satisfies that
	$$\|D_xf^{k}|_{E(x)}\| \leq e^{\lambda_2}, ~~\|D_{f^k(x)}f^{-k}|_{F(f^k(x))}\|\leq e^{-\lambda_1},$$
	then we have
	$$\angle(E(x),F(x))\geq  \frac{e^{\lambda_1}-e^{\lambda_2}}{\|Df\|_{\sup}^k}.$$
\end{Lemma}
\begin{proof}
	Take $v_{E}\in E(x)$ and $v_{F} \in F(x)$ be unit vectors such that the angle between $v_{E}$ and $v_{F}$ realizes the angle between $E(x)$ and $F(x)$. 
	Thus, one has
	\begin{align*}
		0<e^{\lambda_1}-e^{\lambda_2}&\leq \|D_{x}f^{k}(v_F)\|-\|D_{x}f^{k}(v_E)\| \leq \| D_{x}f^{k}(v_F-v_E) \| \\
		&\leq \|D f\|_{\sup}^{k} \|v_F-v_E\| \\
		&=2\|D f\|_{\sup}^{k} \sin \Big(\frac{1}{2}\angle\big(E(x),F(x)\big)\Big) \\
		&\leq \|D f\|_{\sup}^{k} \cdot \angle(E(x),F(x)).
	\end{align*}
	This completes the proof of the Lemma.
\end{proof}

\begin{proof}[Proof of  Theorem~\ref{Thm:Quasi-Diffeo}]
	Assume that $f$ is a $C^\infty$ diffeomorphism with positive topological entropy. 
	Fix constants $1>\alpha_1>\alpha_2>\frac{1}{2}$ and $0<\chi_0<\chi<\chi(\alpha_1,\alpha_2)$. 
	Recall the definition of $\Lambda^{u,\chi}_{N,\beta}(g^q)$ in \eqref{eq:Lambda}; the $C^r$ neighborhood $\mathcal{U}$ of $f$ in \eqref{eq:U}; and $\beta_0>0$, $C>0$ fulfill the condition in \eqref{eq:Keyset-u}. 
	By Proposition \ref{Prop:KeyProp}, we have
	$$\mu(\Lambda^{u,\chi}_{N,\beta}(g^q))=\mu \big(\{x: \exists \, 0\leq i\leq N~\text{s.t.}~g^{qi}(x)\in L^u_{\beta}(g)\cap H_{\chi}^{u}(g^q)\} \big)>\alpha_2$$
	and 
	$$\mu(\Lambda^{s,\chi}_{N,\beta}(g^q))=\mu \big(\{x: \exists \, 0\leq i\leq N~\text{s.t.}~g^{-qi}(x)\in L^s_{\beta}(g)\cap H_{\chi}^{s}(g^q)\} \big)> \alpha_2.$$
	Therefore, by the invariance of $\mu$, it follows that
	$$\mu\Big(g^{-q} \big( \Lambda^{u,\chi}_{N,\beta}(g^q) \big)\cap \Lambda^{s,\chi}_{N,\beta}(g^q)\Big)> 2\alpha_2-1.$$
	We denote $\Lambda'= g^{-q} ( \Lambda^{u,\chi}_{N,\beta}(g^q) )\cap \Lambda^{s,\chi}_{N,\beta}(g^q)$, then
	$$\Lambda'= \{x: \exists \, 0\leq i,j\leq N~\text{s.t.}~g^{-qi}(x)\in L^s_{\beta}(g)\cap H_{\chi}^{s}(g^q),~g^{q(j+1)}(x)\in L^u_{\beta}(g)\cap H_{\chi}^{u}(g^q)\}.$$
	For each $x\in \Lambda'$, there exist $i\in \{0,\cdots,N\}$ and $j\in \{1,\cdots,N+1\}$ such that
	$$ \|D_{g^{-qi}(x)}g^{q(j+i)}|_{E^s(g^{-qi}(x))}\|\leq e^{-q(j+i)\chi},~~\|D_{g^{qj}(x)}g^{-q(j+i)}|_{E^u(g^{qj}(x))}\|\leq e^{-q(j+i)\chi}.$$
	By Lemma \ref{Lem:angles}, we have
	$$\angle\big(E^s(g^{-qi}(x)),E^u(g^{-qi}(x))\big)\geq \frac{e^{q(i+j)\chi}-e^{-q(i+j)\chi}}{\|Dg^q\|^{i+j}_{\sup}} \geq   \frac{e^{q\chi}-e^{-q\chi}}{(2\|Df\|_{\sup})^{q(2N+1)}} >0.$$
	Therefore, for every $x\in \Lambda'$ we have 
	\begin{align*}
		\sin \angle(E^u(x),E^s(x))&\geq (\|Dg\|\cdot \|Dg^{-1}\|)^{-qi} \cdot \sin\angle\big(E^{u}(g^{-qi}(x)),E^{u}(g^{-qi}(x))\big) \\
		&\geq (2\|Df\|\cdot \|Df^{-1}\|)^{-qN} \cdot \sin \left(\frac{e^{q\chi}-e^{-q\chi}}{(2\|Df\|_{\sup})^{q(2N+1)}} \right)>0.
	\end{align*}
	Choose $\beta_0>0$, and enlarge $C:=C(N,\chi,\chi_0,q,\mathcal{U})$ if necessary, so that for every $g\in \mathcal{U}$ and every ergodic measure $\mu$ of $g$ with $h_{\mu}(g)\geq \alpha_1h_{\rm top}(f)$, one has (recall the Definition \ref{Def:Quasi-Hyperbolic} for the set $\mathcal{H}^{\chi_0}_{C}(g)$ )
	$$\Lambda' \subset L^u_{\beta_0}(g)\cap L^s_{\beta_0}(g)\cap \mathcal{H}^{\chi_0}_{C}(g).$$
	Since $\mu(\Lambda')> 2\alpha_2-1$, this completes the proof of Theorem~\ref{Thm:Quasi-Diffeo}.
\end{proof}

\subsection{Measure estimates for Pesin Sets in surface diffeomorphisms}
Now, we prove Theorem \ref{Thm:SPR-Diff}, which can be almost directly deduced from Theorem \ref{Thm:Quasi-Diffeo}.
\begin{proof}[Proof of Theorem \ref{Thm:SPR-Diff}]
	Let $f: M\rightarrow M$ be a $C^{\infty}$ surface diffeomorpism with positive topological entropy.  
	Let $\chi:= \frac{1}{3} h_{\rm top}(f)$ and let $\alpha_1=1-\tau$, $\alpha_2=1-2\tau$. 
	By construction, we have that 
	$$\chi<\chi(\alpha_1,\alpha_2)=\frac{1}{2} h_{\rm top}(f).$$
	Let $C_{\chi}:=8C(\chi,f)$, where $C(\chi,f)$ is the constant in Lemma \ref{Lem:Pes-QH}. 
	By Ruelle's inequality \cite{Rue78}, each ergodic measure $\mu$ of a surface diffeomorpism with positive entropy has exactly one positive and one negative Lyapunov exponents.
	By Theorem~\ref{Thm:Quasi-Diffeo}, there exists a constant $\ell>0$ and a $C^{\infty}$ neighborhood $\cU$ of $f$ such that for every $g\in \cU$ and for every ergodic measure $\mu$ of $g$ with $h_{\mu}(g)\geq (1-\tau) h_{\rm top}(f)$, one has
	$$\mu(\mathcal{H}_{\ell}^{\chi}(g))>1-4\tau.$$
	Then, by Lemma \ref{Lem:Pes-QH}, for each $\varepsilon>0$ we have
	\begin{align*}
		\mu ({\rm PES}_{\ell}^{\chi,\varepsilon}(g))
        &>1-\frac{C(\chi,g)}{\varepsilon}(1-\mu(\mathcal{H}_{\ell}^{\chi}(g)) \\
		&>1- \frac{2C(\chi,f)}{\varepsilon}(1-1+4\tau)\\
        &=1- \frac{C_{\chi}}{\varepsilon} \tau. 
	\end{align*}
	This completes the proof of Theorem~\ref{Thm:SPR-Diff}.
\end{proof}
     
\subsection{Proof of Proposition \ref{Prop:KeyProp}}
We give a lower bound for the measure of $\Lambda^{u,\chi}_{N,\beta}(g^q)$; the lower bound for the measure of $\Lambda^{s,\chi}_{N,\beta}(g^q)$ can be obtained by considering $g^{-1}$. 
For convenience, we denote $G:=g^q$ and 
\begin{equation}\label{eq:Lambdas}
	\Lambda:=\Lambda^{u,\chi}_{N,\beta}(g^q)=\{x: \exists \, 0\leq i\leq N~\text{s.t.}~g^{qi}(x)\in L^u_{\beta}(g)\cap H_{\chi}^{u}(G)\}.
\end{equation}
Since $\mu$ may not to be ergodic of $G$, we consider the ergodic component $\nu$ of $\mu$ with respect to $G$. 
Note that $\mu=\frac{1}{q}\sum_{j=0}^{q-1}g_*^j\nu$.
We want to show that 
\begin{equation}\label{eq:key}
	h_{\mu}(g)=\big(h_{\nu}(g^q)/q\big)\lessapprox \nu(\Lambda)h_{\rm top}(f)+(1-\nu(\Lambda))\chi.
\end{equation}
in the following Proposition.

\begin{Proposition}  \label{Prop:KeyProp-ergodic case} 
     Recall the notation introduced at the beginning of this section. For every ergodic component $\nu$ of $\mu$ with respect to $G$, we have
     	   $$h_{\nu}(G)\leq \nu(\Lambda)q(h_{\rm top}(f)+\alpha)+\big(1-\nu(\Lambda)\big)\cdot q\chi+q\alpha+\frac{1}{r-1}\Big(\int \log \|D_xG\|~{\rm d}\nu(x)-\lambda^{+}(\nu,G)\Big).$$
 \end{Proposition}

 \begin{proof}[Proof of Proposition \ref{Prop:KeyProp}]
  Consider the ergodic decomposition of $\mu=\frac{1}{q}\sum_{j=0}^{q-1} \nu_j$ with respect to $g^{q}$, where $\nu_j=g_*^j\nu$. 
  Note that $h_{\nu_j}(G)=qh_{\mu}(g)$, $\lambda^{+}(\nu_j,G)=q\lambda^{+}(\mu,g)$ for every $0\leq j <q$.
  By Proposition \ref{Prop:KeyProp-ergodic case}, for every $0\leq j<q$ we have that 
  $$h_{\mu}(g)\leq \nu_{j}(\Lambda)(h_{\rm top}(f)+\alpha)+\big(1-\nu_{j}(\Lambda)\big)\cdot \chi+\alpha+\frac{1}{r-1}\Big(\frac{1}{q}\int \log\|D_xg^q\|~{\rm d}\nu(x)-\lambda^{+}(\mu,g)\Big).$$
  Therefore, for every $0\leq j<q$ we have 
  $$\nu_{j}(\Lambda)\geq \frac{h_{\mu}(g)-\chi-\frac{1}{r-1}\big(\frac{1}{q}\int \log \|D_x g^q \|~{\rm d}\mu(x)-\lambda^{+}(\mu,g)\big)-\alpha}{h_{\rm top}(f)-\chi+\alpha}.$$
  By taking the sum over $j$ and divide $q$, we have
   \begin{equation}\label{eq:m-Lambda}
   \mu(\Lambda)\geq  	\frac{h_{\mu}(g)-\chi-\frac{1}{r-1}\big(\frac{1}{q}\int \log \|D_x g^q \|~{\rm d}\mu(x)-\lambda^{+}(\mu,g)\big)-\alpha}{h_{\rm top}(f)-\chi+\alpha}.
   \end{equation}
    This completes the proof of Proposition \ref{Prop:KeyProp}. 
\end{proof}
   
    \subsection{Proof of Proposition \ref{Prop:KeyProp-ergodic case}}
    \subsubsection{Choice of the uniform set $K$ and a strongly $\varepsilon$-bounded curve $\sigma$}\label{Subsec:choose-K-sigma}
     Recall that $G=g^q$ and $\nu$ is an ergodic measure of $G$. 
     Let $K$ be as in Proposition \ref{Prop:Two Balls}. We may find a compact subset $K'$  of $K$ satisfying  the following properties:
     \begin{enumerate}
     	\item[(1)] $\nu(K')>0$ and $K' \subset H_{\chi}^u(G)$, where $H_{\chi}^u(G)$ is defined in  \eqref{eq:lambda};
     	\item[(2)] $\frac{1}{n} \# \{0\leq i < n: G^{i}(x)\in \Lambda\}\rightarrow \nu( \Lambda)$ uniformly in $x\in K'$;
     	\item[(3)] $\frac{1}{n}\log \|D_xG^{n}|_{E^u (x)}\|\rightarrow \lambda^{+}(\nu,G)$ uniformly in $x\in K'$;
     	\item[(4)] $\frac{1}{n} \sum_{i=0}^{n-1} \delta_{G^{i}(x)} \rightarrow \nu $ uniformly in  $x\in K'$.
     \end{enumerate}
     Fix a point $x_0\in K'$ and choose a strongly $\varepsilon$-bounded curve $\sigma:[-1,1]\rightarrow W^u(x_0)$ with $\sigma(0)=x_0$ and  
     $\nu_{\xi(x_0)}( \sigma_{\ast}\cap K')>0$. 
     We consider a sequence of finite partitions $\{\cP_k\}_{k\in \NN}$, satisfying $\nu(\partial \cP_k)=0$ and  ${\rm Diam}(\cP_k)\xrightarrow{k\to \infty} 0$.   
     Then, by Proposition \ref{Prop:Two Balls} we have with $\cP_k^n:=\cP_k^{n,G}$
    \begin{equation}\label{e.entropy-cardinality}
    h_{\nu}(G)\leq \liminf_{k\rightarrow \infty} \left(\liminf_{n\rightarrow \infty}\frac{1}{n} \log \# \{ P\in \cP_k^n : P\cap K' \cap \sigma_*\neq \emptyset\} \right).
    \end{equation}
    To simplify the notations we denote in the following  by $K$ its subset $K'$ satisfying the above properties.
    
    \begin{Proposition}\label{Pro:entropy-bound-reparametrization}
    	Assume that there is a sequence of affine reparametrizations $\{\Gamma_n\}_{n>0}$ such that for any $n\in\mathbb N$,
    	\begin{itemize}
    		\item  $K\cap\sigma_*\subset \bigcup_{\gamma\in\Gamma_n}\sigma\circ\gamma([-1,1])$;
    		\item  for any $\gamma\in \Gamma_n$, one has $G^{j}\circ\sigma\circ\gamma$ is strongly $2\varepsilon$-bounded for any $1 \le j\le n$;
    	\end{itemize}
    	Then, we have $h_\nu(G)\le \liminf\limits_{n\to\infty}\frac{1}{n}\log\#\Gamma_n$.
    \end{Proposition}
    \begin{proof}
    	We take a partition $\mathcal{P}\in \{\cP_k\}_{k\in \NN}$ with small diameter and $\nu(\partial\mathcal P)=0$. Define
    	$$D_{\cal P}(n):=\sup_{\gamma\in \Gamma_n}\#\{P:~P\cap K\cap (\sigma\circ\gamma)_*\neq\emptyset, ~P\in{\cal P}^n\}.$$
    	As in \cite[Proposition 5.2]{LuY24}, one has that $\lim\limits_{n\to\infty}\frac{1}{n}\log D_{\cal P}(n)=0$.
    	Then, as in  \cite[Section 6]{LuY24},
    	$$\#\{P:~P\cap K\cap\sigma_*\neq\emptyset,~P\in{\cal P}^n\}\le \#\Gamma_n \cdot \sup_{\gamma\in \Gamma_n}\{P:~P\cap K\cap (\sigma\circ\gamma)_*\neq\emptyset, ~P\in{\cal P}^n\}.$$
    	Thus, by applying \eqref{e.entropy-cardinality}, one can conclude:
    	$$h_\nu(G)\le \liminf_{n\to\infty}\frac{1}{n}\log \#\Gamma_n + \liminf_{\mathrm{Diam}({\cal P} ) \rightarrow  0} \lim_{n\to\infty}\frac{1}{n}\log D_{\cal P}(n)=\liminf_{n\to\infty}\frac{1}{n}\log \#\Gamma_n.$$
    	This completes the proof.
    \end{proof}
    In the following subsubsections, we will construct affine reparametrizations that satisfy the conditions of Proposition \ref{Pro:entropy-bound-reparametrization}, and provide an upper bound estimate on their number.

     \subsubsection{Decomposition $K$ by red-Bowen balls}
     For every $x$ and $n>0$, let (note that $L_{\beta}^u(g)=L_{\beta}^u(g^q)$)
     \begin{align*}
     	&T_{n}(x):=\{0< i\leq n: G^{i}(x)\in L_{\beta}^u(G)\cap H_{\chi}^u(G)\}
     \end{align*}
     and let 
     $$\chi^{+}(x):=\lceil \log \|D_xG\| \rceil,~~\chi(x):=\lceil \log \|D_xG|_{E^u(x)}\| \rceil.$$
     For $E_n \subset \{1,\cdots,n\}$ and $2n$-integers $\{\chi^{+}_i,\chi_i\}_{i=1}^{n}$, we define the set 
     \begin{equation}\label{e.define-SIgma}
     	\begin{aligned}
     		&\Sigma(E_n,~\{\chi^{+}_i,\chi_i\}_{i=1}^{n}):=\\
     		             &\left\{x\in K\cap \sigma_{\ast}:T_{n}(x)=E_n,~ \chi^{+}(G^{(j-1)}(x))=\chi^{+}_j,~\chi(G^{(j-1)}(x))=\chi_j,~1\leq j\leq n \right\}.
     	\end{aligned}
     \end{equation}
     By a direct computation, we know that $K$ can be covered by at most $(2qA_{g})^{2n}$ number of the sets of the form 
     $\Sigma(E_n,\{\chi^{+}_i,\chi_i\}_{i=1}^{n})$ (see \eqref{eq: Q} for the definition of $A_g$).     
     Now, we fix a subset $\Sigma:=\Sigma(E_n,\{\chi^{+}_i,\chi_i\}_{i=1}^{n})$. 
     For each $n>0$, we consider  
     $$ E^{N}_n:=\{k \in [1,n]: \exists \, 0 \leq i \leq N~\text{s.t.}~k+i \in E_n \, \}. $$
     It follows from the definitions that if $\Sigma(E_n, \{\chi^{+}_i,\chi_i\}_{i=1}^{n})\neq \emptyset ~\text{for some}~\big(\{\chi^{+}_i,\chi_i\}_{i=1}^{n}\big)$ then there is a point $x\in K $ such that $E_n^N=\{0\leq i<n : G^{i}(x)\in \Lambda\}$. 
     Then, by property (2) of Subsection \ref{Subsec:choose-K-sigma} we have 
      \begin{equation}\label{eq:time}
     	\begin{aligned}
     		\lim_{n\rightarrow \infty}  \left(\sup\Big \{\frac{1}{n}\cdot \big| \# E^{N}_n-n\cdot \nu(\Lambda) \big|:~\Sigma(E_n, \{\chi^{+}_i,\chi_i\}_{i=1}^{n})\neq \emptyset ~\text{for some}~\left(\{\chi^{+}_i,\chi_i\}_{i=1}^{n}\right) \Big \}\right)=0
     	\end{aligned}
     \end{equation}
     For convenience, we denote $E^{N}_n$ as the union of pairwise disjoint intervals\footnote{The last interval $(a_m,b_m]$ may happens $b_m=n \notin E_n$ and $b_m-a_m<N$, but this does not affect the final asymptotic estimate of the number of affine reparametrizations, so we omit this case.}
     $$ E^{N}_n:=\bigcup_{j=1}^{m} (a_j,b_j]:~~0\leq a_j<b_j\leq n,~b_j-a_j\geq N,~b_j\in E_n,~\forall j=1,\cdots,m,~~m\leq \frac{n}{N}+1.$$    
    We define the red-$(n,\varepsilon,G)$-Bowen balls\footnote{In \cite[Section 4]{Bur24}, ``red" refers to a geometric time, that is, a time at which the point has a long unstable manifold and the derivative satisfies the Pliss condition.} with respect to the set $\Sigma$ by
    \begin{equation}\label{e.red-bowen-ball}
     B^{\rm red}_{G}(x,n,\varepsilon):=\{y \in\Sigma :d(G^{j}(x),G^{j}(y))\leq \varepsilon,~j\in E^{N}_n \}.
    \end{equation}
    Then, the set $\Sigma$ can be covered by at most
    \begin{equation}\label{e.number-red}
    \exp \left( \# E^{N}_n \cdot q \big(h_{\rm top}(f)+\alpha \big) \right)
    \end{equation}
     red-$(n,\varepsilon,G)$-Bowen balls. 
     Indeed, each interval $(a_j,b_j]$ requires at most $r(G,b_j-a_j,\varepsilon)$ numbers of $(b_j-a_j,\varepsilon)$-Bowen balls to cover the whole manifold.
     It follows from  the choice of $N$ in \eqref{eq:N} that 
     $$r\left(G,b_j-a_j,\frac{\varepsilon}{2}\right)\leq \exp \left( (b_j-a_j) \cdot q(h_{\rm top}(f)+\alpha)\right).$$
     Therefore  the total number of red-$(n,\varepsilon,G)$-Bowen balls needed to cover $\Sigma$ is bounded above by 
     $$\exp \left( \sum_{j=1}^{m}(b_j-a_j) q\cdot (h_{\rm top}(f)+\alpha) \right) \leq \exp \left( \# E^{N}_n \cdot q(h_{\rm top}(f)+\alpha) \right).$$
     
     \subsubsection{Reparametrizations for red-Bowen balls}
     Red-Bowen balls are larger than the usual Bowen balls. 
     However, one can still control the number of parametrizations required to cover red-Bowen balls, as in Proposition \ref{Pro:entropy-bound-reparametrization}.          
     Recall that $\sigma$ is a strongly $\varepsilon$-bounded curve whose image is contained in $W^u_{\rm loc}(x_0)$, and that $\varepsilon$ satisfies $10 \cdot \varepsilon \leq \varepsilon_{\Upsilon}$, where  $\varepsilon_{\Upsilon}$ satisfies the conclusion of Lemma \ref{Lemma:local-reparametrization} and $\|D^jG\|_{\sup}<\Upsilon$ for any $1\leq j \leq r$. 
     \begin{Theorem}\label{Thm:red-bowen-reparametrization}
     	Given $\Sigma:=\Sigma(E_n,\{\chi^{+}_i,\chi_i\}_{i=1}^{n})\subset K$, for each $x\in \sigma_* \cap \Sigma$, there exists a family $\Gamma_n(\Sigma,x)$ of affine reparametrizations such that
     	$$\# \Gamma_n(\Sigma,x)\leq (2C_r)^n \cdot e^{ (n-\#E^{N}_n) \cdot q\chi } \cdot  \exp\left( \dfrac{\sum_{i=1}^{n} (\chi_i^{+}-\chi_i)}{r-1} \right),~~\text{and}$$
     	\begin{enumerate}
     		\item [(1)]  $B^{\rm red}_{G}(x,n,\varepsilon) \subset \bigcup_{\gamma\in \Gamma_n(\Sigma,x)}  \sigma\circ \gamma([-1,1])$;
     		\item [(2)]  $G^{j}\circ \sigma \circ \gamma$ is strongly $2\varepsilon$-bounded for every $j=1,\cdots,n$.
     	\end{enumerate}
     \end{Theorem}
     \begin{proof}
     	Recall that  
     	     $$E^{N}_n:=\bigcup_{j=1}^{m} (a_j,b_j],~~0\leq a_j<b_j<n,~b_j-a_j\geq N,~b_j\in E_n,~\forall j=1,\cdots,m.$$
     	We will show by induction that for every $1\leq k \leq  n$, there exists a family $\Gamma_k$ of affine reparametrizations satisfying (Item (4) serves as the inductive step)
     	\begin{enumerate}
     		\item [(1)]  $\big\{y \in\Sigma :d(G^{j}(x),G^{j}(y))\leq \varepsilon,~j\in E^{N}_n\cap (0,k] \big\} \subset \bigcup_{\gamma\in \Gamma_k}  \sigma\circ \gamma([-1,1])$;
     		\item [(2)]  $G^{j}\circ \sigma \circ \gamma$ is strongly $2\varepsilon$-bounded for every $j=1,\cdots,k$ and every $\gamma \in \Gamma_k$;
     		\item [(3)]  $\# \Gamma_k\leq (2C_r)^{k} \exp\left(\frac{\sum_{i=1}^{k} (\chi^{+}_i-\chi_i)}{r-1}\right) \cdot e^{(k-\# (E^{N}_n\cap (0,k]))\cdot q\chi}$;
     		\item [(4)]  for each $\gamma_k\in \Gamma_k$, there exists $\gamma_{k-1}\in \Gamma_{k-1}$ and a reparametrization $\theta:[-1,1]\rightarrow [-1,1]$ such that $\gamma_k=\gamma_{k-1} \circ \theta$, and $|\theta'|\leq  \frac{1}{2} \cdot e^{-q\chi} \cdot \exp\left(\frac{\chi_{k}-\chi^{+}_{k}-1}{r-1}\right)$ whenever $k\notin E^{N}_n$.
     	\end{enumerate}
      Once we prove this induction, we get the conclusion for $n$, and then we conclude the theorem.
      We should check the case of $k=1$ by starting the induction step. 
      However, this is too similar to the case of checking $k+1$ by knowing the case of $k$. 
      Thus, this process is omitted.    
      Assuming the conclusion holds for the case of $k$. 	 
      For each $\gamma\in \Gamma_k$, we denote
      $$\sigma_k:=\sigma_{k,\gamma}=G^{k} \circ \sigma \circ \gamma.$$
      By the induction, $\sigma_{k}$ is strongly $2\varepsilon$-bounded.  
      For $k+1$, there are two cases: $k+1\in E^{N}_n$ and $k+1 \notin E^{N}_n$.
      We will construct affine reparametrizations for all such $\sigma_{k}$\footnote{For the initial step $k=1$, we will construct affine reparametrizations for $\sigma_0=\sigma$.}.
 
    We apply Lemma \ref{Lemma:local-reparametrization} for $\sigma_{k}$, there exists a family of affine reparametrizations $\Theta(\sigma_{k})$ such that
    \begin{enumerate}
    	\item[(a)] $\sigma_{k}([-1,1]) \cap G^{k}(\Sigma) \subset \bigcup_{\theta \in \Theta(\sigma_{k})} \sigma_{k} \circ \theta([-1,1])$;
    	\item[(b)] $G \circ  \sigma_{k} \circ \theta$ is bounded for every $\theta \in \Theta(\sigma_{k})$;
    	\item[(c)] $\# \Theta(\sigma_{k})\leq C_r \exp \big( \frac{\chi_{k+1}^{+}-\chi_{k+1}}{r-1}\big)$;
    	\item[(d)] $|\theta'|\leq \frac{1}{2} \cdot \exp \big(\frac{\chi_{k+1}-\chi_{k+1}^{+}-1}{r-1}\big)$.
    \end{enumerate}
     
     If $k+1\in E^{N}_n$, then we need to cover the set
     $$\{z\in \sigma_{k}([-1,1])\cap G^{k}(\Sigma): d(G(z),G^{k+1}(x))<\varepsilon\}.$$
     For every $\theta \in \Theta(\sigma_{k})$, since $G \circ  \sigma_{k} \circ \theta$ is bounded and, by \eqref{eq:U}, its length is smaller than the injectivity radius of the exponential map, the affine map $\gamma(\theta)$ from $[-1,1]$ onto the interval
     $$\big[\min\{t\in [-1,1]: G \circ  \sigma_{k} \circ \theta(t)\in B(G^{k+1}(x),\varepsilon)\},~\max\{t\in [-1,1]: G \circ  \sigma_{k} \circ \theta(t)\in B(G^{k+1}(x),\varepsilon)\}\big]$$
     satisfies that $G \circ  \sigma_{k} \circ \theta \circ \gamma(\theta)$ is strongly $2\varepsilon$-bounded.
     Now we take
     $$\Gamma_{k+1}:=\left\{\gamma \circ \theta \circ \gamma(\theta) :~\gamma\in\Gamma_k,~\theta\in\Theta(\sigma_{k,\gamma})\right\}.$$
     We check that $\Gamma_{k+1}$ satisfies the induction. 
     The Items (1) (2) (4) are given by the construction.
     We check the cardinality, if $k+1\in E^{N}_n$, then $k+1-\# ( E^{N}_n\cap (0,k+1])=k-\# ( E^{N}_n\cap (0,k])$ 
\begin{align*}
\#\Gamma_{k+1}&\le\#\Gamma_k\times\sup_{\gamma\in\Gamma_k}\#\Theta(\sigma_{k,\gamma})\\
&\leq (2C_r)^{k} \exp\left(\frac{\sum_{i=1}^{k} (\chi^{+}_i-\chi_i)}{r-1}\right) \cdot e^{(k-\# (E^{N}_n\cap [1,k]))\cdot q\chi} \cdot  C_r \cdot \exp \left( \frac{\chi_{k+1}^{+}-\chi_{k+1}}{r-1}\right)\\
& \le(2C_r)^{k+1} \cdot e^{(k+1-\# (E^{N}_n\cap [1,k+1])) \cdot q\chi}  \cdot    \exp\left(\frac{\sum_{i=1}^{k+1} (\chi^{+}_i-\chi_i)}{r-1}\right).
\end{align*}
This completes the construction for the case $k+1\in E^{N}_n$.

Now we consider the case that $k+1\notin E^{N}_n$.
This implies that $G^{k+1}(y)\notin H_{\chi_0}^{u} (G)\cap L^u_{\beta}(g)$ for any $y\in\Sigma$.
Let $\Theta_{\chi}$ be a family of affine maps from $[-1,1]$ to $[-1,1]$ and satisfies
\begin{itemize}
	\item $|\theta'|\leq e^{-q\chi}$ for every $\theta\in \Theta_{\chi}$;
	\item $[-1,1]=\bigcup_{\theta\in \Theta_{\chi}} \theta([-1,1])$ and $\# \Theta_{\chi}\leq e^{q\chi}$;
\end{itemize}

Consider $\Theta_{\chi}(\sigma_{k}):=\{\gamma_0 \circ \theta_{\chi}: \gamma_0 \in \Theta(\sigma_{k}),~\theta_{\chi} \in \Theta_{\chi} \}$.
For each $\theta\in  \Theta_{\chi}(\sigma_{k})$, we consider the bounded curve $G\circ \sigma_{k} \circ \theta$. 
Assume that $k+1\in (b_{i},a_{i+1}]$ for some $1\leq i \leq m$, where $b_i\in E_n$.
Therefore, by induction there exist affine reparametrizations $\{\gamma_{b_i}, \theta_{b_i+1}, \cdots, \theta_{k+1}\}$ such that 
\begin{enumerate}
	\item  $\gamma \circ \theta = \gamma_{b_i} \circ \theta_{b_i+1} \circ \cdots \circ \theta_{k+1}$;
	\item  $\gamma_{b_i} \in \Gamma_{b_i}$ and $|\theta_{j}'|\leq e^{-q\chi} \cdot \exp\left(\frac{\chi_{j}-\chi_j^{+}-1}{r-1}\right)$ for any $j\in (b_i,k+1]$.
\end{enumerate}
By the definition of $\Sigma$ (see \eqref{e.define-SIgma}), $G^{b_i}(z)\in H_{\chi}^u(G)$ for all $z\in \Sigma$.
If $G\circ \sigma_{k} \circ \theta$ is not strongly $2\varepsilon$-bounded, then by Lemma \ref{Lem:Blue} for every $z\in \Sigma \cap G^{-k}( \sigma_{k}\circ \theta([-1,1]))$ we have 
    $$\log\|D_{G^{k+1-j}(z)}G^{j}|_{T_{G^{k+1-j}(z)}(G^{j}\circ \sigma)_{\ast}} \|\geq j\cdot q\chi,~\forall 1\leq j \leq  k+1-b_i.$$
Consequently, for each $y \in G^{k+1}(\Sigma )\cap (G\circ \sigma_{k}\circ \theta)([-1,1])$, since the above inequality holds for $G^{-(k+1)}(y)$ and $G^{-(k+1)+b_i}(y)\in  H_{\chi}^u(G)$, we obtain $y\in H_{\chi}^u(G)$. 
Hence $y\notin L^u_{\beta}(g)$, because $k+1\notin E_n$. 
Therefore, by Lemma \ref{Lem:2-Res} there exists two affine reparametrizations $\gamma_{\theta}^{\pm}$ such that 
\begin{itemize}
	\item $G \circ \sigma_{k} \circ \theta \circ \gamma_{\theta}^{\pm}$ is strongly $2\varepsilon$-bounded;
	\item $G^{k+1}(\Sigma )\cap (G\circ \sigma_{k}\circ \theta)([-1,1])\subset G\circ\sigma_{k}\circ \theta \circ \gamma_{\theta}^{+} ([-1,1])\bigcup  G\circ \sigma_{k}\circ \theta \circ \gamma_{\theta}^{-} ([-1,1])$.
\end{itemize}    
 If $G\circ \sigma_{k} \circ \theta$ is strongly $2\varepsilon$-bounded, we take $\gamma_{\theta}^{\pm}=Id$.
 Now, we construct $\Gamma_{k+1}$ as
 $$\Gamma_{k+1}:=\left\{\gamma \circ \theta \circ \gamma_{\theta}^{\pm} :~\gamma \in\Gamma_k,~\theta\in\Theta_{\chi}(\sigma_{k,\gamma})\right\}.$$
 We check that $\Gamma_{k+1}$ satisfies the induction. 
 The Items (1) (2) (4) are given by the construction.
 We check the cardinality, in this case, $k+1-\# ( E^{N}_n\cap (0,k+1])=1+k-\# ( E^{N}_n\cap (0,k])$, 
 \begin{align*}
 	\#\Gamma_{k+1}&\le\#\Gamma_k\times\sup_{\gamma\in\Gamma_k}\#\Theta_{\chi}(\sigma_{k,\gamma}) \times 2\\
 	&\leq (2C_r)^{k} \exp\left(\frac{\sum_{i=1}^{k} (\chi^{+}_i-\chi_i)}{r-1}\right) \cdot e^{(k-\# (E^{N}_n\cap (0,k]))\cdot q\chi} \cdot e^{q\chi} \cdot 2C_r \cdot \exp \left( \frac{\chi_{k+1}^{+}-\chi_{k+1}}{r-1}\right)\\
 	& \le(2C_r)^{k+1} \cdot e^{(k+1-\# (E^{N}_n\cap (0,k+1])) \cdot q\chi}  \cdot  \exp\left(\frac{\sum_{i=1}^{k+1} (\chi^{+}_i-\chi_i)}{r-1}\right).
 \end{align*}
 This completes the construction for the case $k+1\notin E^{N}_n$.  
\end{proof}
       
\subsubsection{End of the proof of Proposition~\ref{Prop:KeyProp-ergodic case}}

Consider an ergodic measure $\nu$ as in Proposition~\ref{Prop:KeyProp-ergodic case}.
Now we give the bound of the numbers of affine reparametrizations.

\begin{Proposition}\label{Pro:number-reparametrizations}
There is a sequence of affine reparametrizations $\{\Gamma_n\}_{n>0}$ that covers $K\cap\sigma([-1,1])$ such that  $G^{j}\circ\sigma\circ\gamma$ is strongly $2\varepsilon$-bounded for any $1 \le j\le n$ and any $\gamma\in\Gamma_n$, and
$$\limsup_{n\to\infty}\frac{1}{n}\log\#\Gamma_n\le \nu(\Lambda) \cdot q (h_{\rm top}(f)+\alpha) +q\alpha+(1-\nu(\Lambda))q\chi+\frac{\big(\int \log \|DG\|~{\rm d}\nu-\lambda^{+}(\nu,G)\big)}{r-1}.$$
\end{Proposition}

\begin{proof}
To summarize, for sufficiently large $n>0$ we have that 
\begin{itemize}
\item there are at most $(2qA_g)^{2n}$ types  $\Sigma:=\Sigma(E_n, \{\chi^{+}_i,\chi_i\}_{i=1}^{n})$ as defined in \eqref{e.define-SIgma};
\item each $\Sigma(E_n, \{\chi^{+}_i,\chi_i\}_{i=1}^{n})$ can be coverd by at most 
$$\exp \left( \# E_{n}^N\cdot q\left(h_{\rm top}(f)+\alpha\right)\right)$$ 
numbers of red-$(n,\varepsilon,G)$-Bowen balls;
\item each red-$(n,\varepsilon,G)$-Bowen ball can be covered by at most
$$(2C_r)^n e^{(n-\# E_{n}^N)\cdot q\chi}\exp\left(\dfrac{\sum_{i=1}^{n}(\chi_i^{+}-\chi_i)}{r-1}\right)$$
affine reparametrizations satisfies the conclusion of Theorem \ref{Thm:red-bowen-reparametrization}. 
\end{itemize}
Thus, totally we have a family of affine reparametrizations $\Gamma_n$, whose cardinality is upper bounded by
\begin{align*}
	(2qA_g)^{2n} \cdot \sup_{E_n}\exp \left( \# E_{n}^N \cdot q \cdot \big(h_{\rm top}(f)+\alpha \big) \right)  \cdot (2C_r)^n e^{(n-\# E_{n}^N)\cdot q\chi}\exp \left(\dfrac{\sum_{i=1}^{n}(\chi_i^{+}-\chi_i)}{r-1}\right)
\end{align*}
Recall that 
$$\lim\limits_{n\rightarrow \infty}  \sup_{E_n} \Big|\frac{1}{n} \# E_{n}^N - \nu (\Lambda) \Big| = 0.$$
For each $\Sigma(E_n, \{\chi^{+}_i,\chi_i\}_{i=1}^{n})\neq \emptyset$, one can choose $z\in \Sigma(E_n, \{\chi^{+}_i,\chi_i\}_{i=1}^{n})\subset K$ such that 
\begin{align*}
	\limsup_{n\rightarrow \infty} \frac{1}{n}\sum_{i=1}^{n} \big(\chi_i^{+}-\chi_i \big) 
	&\leq   \lim_{n\rightarrow \infty}  1+\frac{1}{n} \big(\sum_{i=0}^{n-1} \log \|D_{G^{i}(z)} G\|- \log \|D_{z} G^{n}|E^{u}(z)\| \big)\\
	&\leq 1+ \int \log \|D_xG\|{\rm d}\nu(x)-\lambda^{+}(\nu,G).
\end{align*}
Then, by taking the limit, one has (recall \eqref{eq: Q} for $q$)
\begin{align*}
	&\limsup_{n\to\infty} \frac{1}{n} \log \# \Gamma_n\\
	 \leq  ~&\nu(\Lambda) \cdot q\big(h_{\rm top}(f)+\alpha\big)+(1-\nu(\Lambda))\cdot q\chi+\frac{\int \log \|DG\|~{\rm d}\nu-\lambda^{+}(\nu,G)}{r-1} 
	+2\log 2qC_rA_g +\frac{1}{r-1} \\
	\leq ~&\nu(\Lambda) \cdot q(h_{\rm top}(f)+\alpha)+(1-\nu(\Lambda))\cdot q\chi +\frac{\int \log \|D_xG\|~{\rm d}\nu(x)-\lambda^{+}(\nu,G)}{r-1}+q\alpha.
\end{align*}
This completes the proof of Proposition \ref{Pro:number-reparametrizations}.
\end{proof}
By Proposition \ref{Pro:entropy-bound-reparametrization} and Proposition \ref{Pro:number-reparametrizations}, we complete the proof of Proposition~\ref{Prop:KeyProp-ergodic case}.

\begin{Remark}[Effective upper bound on the number of MMEs]
{\rm 
		 The number of measures of maximal entropy of a $C^{\infty}$ surface diffeomorphism $f$ can be explicitly controlled by certain dynamical and differential datas associated to $f$ as follows.
		 
		 Let $\alpha_1=1$ and $\alpha_2=\frac{2}{3}$. 
		 We no longer require $\alpha>0$ to be a small parameter, but just choose $\alpha=h_{\rm top}(f)/6$.
		 Take $\chi=h_{\rm top}(f)/4$. 
		 Define $r_f:=r(\|Df^{\pm}\|_{\sup}, h_{\rm top}(f))>0$ be the minimal integers such that 
		 $$\frac{\log \big(2 \cdot \max \{ \|Df\|_{\sup}, \|Df^{-1}\|_{\sup} \}\big)}{r_f-1}\leq \frac{h_{\rm top}(f)}{12}.$$
		 Next, define $q_f:=q(\|Df^{\pm}\|_{\sup}, h_{\rm top}(f))\in \NN$ be the minimal integer such that 
		 $$\frac{1}{q_f}\log (5qC_r A_f)+\frac{1}{q_f(r_f-1)}<\frac{h_{\rm top}(f)}{12},~~\text{where}~A_f:=2\log(1+\|Df\|_{\sup}+\|Df^{-1}\|_{\sup}).$$
		 Let $\varepsilon_f:=\varepsilon(h_{\rm top}(f), \|D^k(f^{\pm q_f})\|:~1\leq  k \leq r_f)$ be the largest real number with 
	     $$\varepsilon_f^{k-1} \cdot \|D^k(f^{q_f})\|_{\sup} \leq \|Df^{q_f}\|_{\sup},~\forall k=2,\cdots, r_f$$ 
	     and set $\beta_{0}= \varepsilon_f/2$.
         Then, we define $N_{f}(\varepsilon)\in \NN$ as the smallest integer such that
	     $$r(n,f,\varepsilon_f)\leq e^{\frac{5}{4} \cdot n \cdot  h_{\rm top}(f)},~\forall n\geq N_{f}.$$
	     By following the above proofs we get for every ergodic measure $\mu$ with maximal entropy
         $$h_{\mu}(f) \leq h_{\rm top}(f) \cdot \left( \frac{5}{4} \cdot \mu(\Lambda_{N_{f},\beta_0}^{u,\chi}(f^{q_f})) +\frac{1}{4}\big(1-\mu(\Lambda_{N_{f},\beta_0}^{u,\chi}(f^{q_f})) \big) +\frac{1}{6} \right),$$ 
         and 
          $$h_{\mu}(f) \leq h_{\rm top}(f) \cdot \left( \frac{5}{4} \cdot \mu(\Lambda_{N_{f},\beta_0}^{s,\chi}(f^{q_f})) +\frac{1}{4}\big(1-\mu(\Lambda_{N_{f},\beta_0}^{s,\chi}(f^{q_f})) \big) +\frac{1}{6} \right),$$ 
         where $\Lambda_{N_{f},\beta_0}^{u,\chi}(f^{q_f})$ and $\Lambda_{N_{f},\beta_0}^{s,\chi}(f^{q_f})$ are defined in \eqref{eq:Lambda} and \eqref{eq:Lambda-s} respectively.
         Recall the angle estimate in Lemma \ref{Lem:angles} and the proof of Theorem \ref{Thm:Quasi-Diffeo}.
         We deduce the existence of
         $$C:=C(h_{\rm top}(f), N_f, \|D^k(f^{\pm q_f})\|:~1\leq  k \leq r_f)~\text{and}~\beta:=\beta(h_{\rm top}(f), N_f, \|D^k(f^{\pm q_f})\|:~1\leq  k \leq r_f)$$
         such that for every ergodic measure $\mu$ with maximal entropy, one has           
         $$\mu\big(L^u(\beta)\cap L^s(\beta) \cap \{x:\angle(E^s(x),E^u(x))>C^{-1}\}\big)>\frac{1}{6}.$$
         Together with Proposition \ref{Prop:FHCDiff} we obtain the following corollary.
         \begin{Corollary}
         	The number of maximal measures of  a $C^\infty$ surface diffeomorphism $f: M\to M$ is less than ${\rm N}$ for some function ${\rm N}$  depending only on $  h_{\rm top}(f)$, $N_f$, $\|D^k(f^{\pm q_f})\|:~1\leq  k \leq r_f$.
         \end{Corollary}
         }
\end{Remark}

\section{Proof of Main Theorems: 3-dimensional vector fields}     
In this section, we always assume that $M$ is a three-dimensional compact Remiannian manifold without boundary.
Let $X: M\rightarrow TM$ be a $C^{\infty}$ vector field without singularities, and let $(\varphi^t)_{t\in \mathbb{R}}$ denote the flow generated by $X$.

Let $f:=\varphi^1$ be the time one map of the flow $(\varphi^t)_{t\in \mathbb{R}}$.
Since $X$ has no singularities, the flow direction is uniformly bounded in the following sense: there exists a constant $C_{X}>0$ such that
$$\|D_{x}f^{\pm n}|_{E^c(x)}\|\leq C_{X}, ~\forall n\in \mathbb{N},~\forall x\in M,$$
where $E^c(x)={\rm span}(X(x))$.
Therefore, for every hyperbolic measure $\mu$ of $(\varphi^t)_{t\in \mathbb{R}}$, one has
\begin{equation}\label{eq:flow-Non}
	\mu\big(\{ x: C_X^{-1}\leq {\rm m}(D_{x}f^n|_{E^c(x)}) \leq \| D_{x}f^n|_{E^c(x)} \| \leq C_X, \forall n\in \ZZ \} \big)=1.
\end{equation}

Take $\frac{1}{2}<\alpha_2<\alpha_1<1$ and  $0<\chi<\chi(\alpha_1,\alpha_2)$ (see \eqref{eq:chi12} for the definition of $\chi(\alpha_1,\alpha_2)$).
By Ruelle's inequality \cite{Rue78}, each ergodic measure $\mu$ of the flow with positive entropy has exactly one positive and one negative Lyapunov exponents.
By Theorem \ref{Thm:Quasi-Diffeo}, there are $\beta>0$ and $C>0$ such that for every ergodic measure $\nu'$ of $\varphi^{1}$ with $h_{\nu'}(\varphi^1)\geq \alpha_1 h_{\rm top}(X)$, one has (see Definition \ref{Def:long-unstable-stable} and Definition \ref{Def:Quasi-Hyperbolic} for notations)
$$\nu'(L^u_{\beta}(\varphi^1)\cap L^s_{\beta}(\varphi^1) \cap \mathcal{H}_{C}^{\chi}(\varphi^1))> 2\alpha_2-1>0.$$ 
Recall the Definition \ref{Def:QuasiHyperflow} of the set $\mathcal{H}_{\ell}^{\chi,\varepsilon}(X)$.
For every $\varepsilon>0$, by the angel estimate (see Lemma \ref{Lem:angles}) and \eqref{eq:flow-Non}, one can choose $\ell > \max\{C,~C_X\}$ such that $\mathcal{H}_{C}^{\chi}(\varphi^1) \subset \mathcal{H}_{\ell}^{\chi,\varepsilon}(X)$.

\begin{proof}[Proof of Theorem \ref{Thm:F-H-F}]
For each $\tau>0$, let $\alpha_1 = \frac{1}{2} + \tau$ and $\alpha_2 = \frac{1}{2} + 2\tau$, and fix the constants $\chi$, $\beta$, and $\ell$ as chosen above.
Then, for every ergodic measure $\mu$ of the flow $\{\varphi^t\}_{t\in \RR}$ with $h_{\mu}(X) \geq \alpha_1 h_{\rm top}(X)$, each ergodic component $\nu$ of $\mu$ with respect to $\varphi^1$ satisfies $h_{\nu}(\varphi^1) \geq \alpha_1 h_{\rm top}(X)$.
Therefore, we have
$$\nu\left(L^u_{\beta}(\varphi^1) \cap L^s_{\beta}(\varphi^1) \cap \mathcal{H}^{\chi,\varepsilon}_{\ell}(X)\right) > 2\alpha_2 - 1 = \tau > 0.$$
It follows that
$$\mu\left(L^u_{\beta}(\varphi^1) \cap L^s_{\beta}(\varphi^1) \cap \mathcal{H}^{\chi,\varepsilon}_{\ell}(X)\right) > \tau > 0.$$
By Proposition \ref{Prop:FHCFlow}, we have
	$$\#\{{\rm HEC}(\mu): \nu~\text{is an ergodic measure with}~h_{\mu}(X)>\left(\frac{1}{2}+\tau\right) h_{\rm top}(X)\}<+\infty.$$
	This completes the proof of Theorem \ref{Thm:F-H-F}.
\end{proof}

\begin{proof}[Proof of Theorem \ref{Thm:SPR-Flow}]
	For $0<\tau<\frac{1}{4}$, we let $\alpha_1=1-\tau$, $\alpha_2=1-2\tau$, $\chi=\frac{1}{3} h_{\rm top}(X)<\chi(\alpha_1,\alpha_2)$ and let $0<\varepsilon \ll \chi$.
    From the preceding discussion, we conclude that for every ergodic measure $\mu$ of $\{\varphi^t\}_{t\in \RR}$ with  $h_{\mu}(X)\geq \alpha_1 h_{\rm top}(X)$, one has 
    $$\mu(L^u_{\beta}(\varphi^1)\cap L^s_{\beta}(\varphi^1) \cap \mathcal{H}^{\chi,\varepsilon}_{\ell}(X))> 1-4\tau.$$    
	Take $C_{\chi}=4C(\chi,X)$, by Lemma \ref{Lem:Pes-QH-Flows}, we have
	$$\mu({\rm PES}_{\ell}^{\chi, \varepsilon}(X))>1- \frac{C(\chi,X)}{\varepsilon}\Big(1-\mu(\mathcal{H}^{\chi,\varepsilon}_{\ell}(X))\Big)>1- \frac{C_\chi}{\varepsilon}\cdot \tau.$$
	This completes the proof of Theorem \ref{Thm:SPR-Flow}.
\end{proof}

\begin{Remark}[\textit{Remark for $C^r$ smoothness}]
	{\rm 
		In this paper, we only state the results in the $C^\infty$ setting. 
		Let $X: M\rightarrow TM$ be a $C^{r},r>1$ three-dimensional vector field without singularities, and let $(\varphi^t)_{t\in \mathbb{R}}$ denote the flow generated by $X$.
		Recall the definition of 
		$$\lambda_{\min}(\varphi):=\min \Big\{ \lim_{n\rightarrow \infty} \frac{1}{n} \log \|D \varphi^n\|_{\sup},~\lim_{n\rightarrow \infty} \frac{1}{n} \log \|D \varphi^{-n}\|_{\sup} \Big\} $$
		Assume that $h_{\rm top}(X)=\alpha(X)\cdot \lambda_{\min}(\varphi)$ with $\alpha(X)>\frac{2}{r+1}$, then Theorem \ref{Thm:F-H-F} admits a version with a complicated constant: for any $\tau>0$, one has
		$$\# \Big \{{\rm HEC}(\nu): \nu~\text{is an ergodic measure with}~h_{\nu}(X)>\big(\frac{r-1}{2r}+\frac{1}{r\alpha(X)}+\tau\big) h_{\rm top}(X) \Big\}<\infty.$$
		This fact can be obtained via Proposition \ref{Prop:KeyProp} and the same arguments as in the $C^{\infty}$ case. 
		Consequently, the number of ergodic measures of maximal entropy is finite whenever $h_{\rm top}(X)>\frac{2}{r+1}\cdot \lambda_{\min}(\varphi)$ for every $r>1$.
		
		In the $C^r$ setting, Theorem~\ref{Thm:SPR-Flow} partially holds (only for sufficiently small $\tau > 0$, rather than all $0<\tau < \frac{1}{4}$) under the additional assumption of entropy continuity, as introduced in \cite{BCS25}.
		A similar result for surface diffeomorphisms can be found in \cite{Ghezal25}.
		}
\end{Remark}
		
\begin{Remark}[\textit{Remark on the continuity of Lyapunov exponents}] \label{Re:CofLE}
	{\rm 
	As a consequence of Theorem \ref{Thm:SPR-Flow}, we can prove the continuity of Lyapunov exponents in the flowing case: 
    
		\medskip
		
	\textit{Let $X: M\rightarrow TM$ be a $C^{\infty}$ three-dimensional vector field without singularities, and let $(\varphi^t)_{t\in \mathbb{R}}$ denote the flow generated by $X$. 
		Assume that $h_{\rm top}(X)>0$, then for every sequence of ergodic measures $\{\mu_n\}_{n>0}$ of $(\varphi^t)_{t\in \mathbb{R}}$ with $\mu_n \rightarrow \mu$ and $h_{\mu_n}(X) \rightarrow h_{\rm top}(X)$, one has
		$\lambda^{+}(\mu_n) \rightarrow \lambda^{+}(\mu)$.}		
        
		\medskip		
		
		Indeed, for each $\tau>0$, by Theorem \ref{Thm:SPR-Flow} there are $\chi>0, ~0< \varepsilon \ll \chi, ~\ell>0$ and $N>0$ such that 
		$$\mu_n({\rm PES}_{\ell}^{\chi,\varepsilon}(X))>1-\tau,~\forall n>N. $$
        Since the set ${\rm PES}_{\ell}^{\chi,\varepsilon}(X)$ is compact, one has $\mu({\rm PES}_{\ell}^{\chi,\varepsilon}(X))>1-\tau.$
        Note that the geometric function  $\psi(x):=\log \|D_x\varphi^1|_{E^u(x)} \| $ is continuous on ${\rm PES}_{\ell}^{\chi,\varepsilon}(X)$.
        Therefore, we can choose a continuous function $\phi: M\rightarrow \RR$ such that 
        $\| \phi\|_{C^0}\leq \|D\varphi^1\|_{\sup}$ and $\phi(x)=\psi(x)$ for every $x\in {\rm PES}_{\ell}^{\chi,\varepsilon}(X)$.
        Therefore, we have 
        \begin{align*}
        	\limsup_{n\rightarrow \infty} |\lambda^{+}(\mu_n)-\lambda^{+}(\mu)|&= \limsup_{n\rightarrow \infty} \Big| \int  \psi(x) {\rm d} \mu_n(x)- \int  \psi(x) {\rm d} \mu(x) \Big| \\
        	&\leq \lim_{n\rightarrow \infty} \Big| \int  \phi~ {\rm d} \mu_n- \int  \phi~ {\rm d} \mu \Big|+2\tau \cdot \|D\varphi^1\|_{\sup}
        	=2\tau \cdot \|D\varphi^1\|_{\sup}.
        \end{align*}
        Since $\tau>0$ is arbitrary, the conclusion follows.
	}		
\end{Remark}

\section*{Acknowledgement}
The authors gratefully acknowledge the support of the Tianyuan Mathematical Center in Northeast China and thank J. Buzzi and O. Sarig for their helpful comments.

\flushleft{\bf David Burguet} \\
\small LAMFA, Universit\'{e} Jules Vernes, 80000 Amiens, France.\\
\textit{E-mail:} \texttt{david.burguet@u-picardie.fr}\\
   
\flushleft{\bf Chiyi Luo} \\
\small School of Mathematics and Statistics, Jiangxi Normal University, 330022 Nanchang, P. R. China.\\
\textit{E-mail:} \texttt{luochiyi98@gmail.com}\\

\flushleft{\bf Dawei Yang} \\
\small School of Mathematical Sciences, Soochow University, 215006 Suzhou, P. R. China.\\
\textit{E-mail:} \texttt{yangdw@suda.edu.cn}\\

\end{document}